\documentclass{ws-m3as}

\usepackage{graphicx}
\usepackage{amsmath,amsfonts,amssymb}
\usepackage{t1enc}
\usepackage[english]{babel}
\usepackage{verbatim}
\usepackage{url}
\usepackage{subcaption}
\usepackage{booktabs}
\usepackage{array}
\usepackage{xcolor}

\usepackage{bm}
\DeclareMathAlphabet{\mathpzc}{OT1}{pzc}{m}{it}

\usepackage{tikz}
\usetikzlibrary{arrows.meta}
\usetikzlibrary{arrows}
\usetikzlibrary{patterns}

\newcommand{\I}{\mathrm I}
\newcommand{\C}{\mathrm C}
\newcommand{\A}{\mathrm A}
\newcommand{\scaling}{\mathcal D}
\newcommand{\foralls}{\forall\,}
\newcommand{\kappaess}{\kappa_{\mathrm{ess}}}
\newcommand{\degree}{\mathpzc{p}}
\newcommand{\smoothness}{\mathpzc{s}}
\newcommand{\kbox}[1]{ {#1^{(k)}}\raisebox{.8em}{}}
\newcommand{\kkbox}[2]{ {#1^{(#2)}}\raisebox{.8em}{}}

\newcounter{assumption}
\newenvironment{assumption}{
\par \refstepcounter{assumption} {\bf Assumption~\arabic{assumption}.}
}{}

\newenvironment{proofof}[1]{
\par {\bf Proof #1.}
}{\hfill\qed}

\begin{document}

\markboth{Jarle Sogn and Stefan Takacs}{Stable discretizations and IETI-DP solvers for the Stokes system in multi-patch IgA}

\catchline{}{}{}{}{}

\title{Stable discretizations and IETI-DP solvers for the Stokes system
in multi-patch Isogeometric Analysis}

\author{Jarle Sogn\footnote{This work was supported by the Austrian Science Fund (FWF): P31048. The second author was also supported
by the bilateral project WTZ BG 03/2019 (KP-06-Austria/8/2019), funded by  OeAD (Austria) and Bulgarian National Science Fund.
This support is gratefully acknowledged.}}

\address{Department of Mathematics, University of Oslo\\
Postboks 1053, Blindern, Oslo 0316, Norway\\
jarlesog@math.uio.no}

\author{Stefan Takacs${}^*$}

\address{Institute of Numerical Mathematics, Johannes Kepler University Linz, \\Altenberger Str. 69, 
4040 Linz,
Austria\\
stefan.takacs@numa.uni-linz.ac.at}

\maketitle

\begin{history}
\end{history}

\begin{abstract}
     We are interested in a fast solver for the  Stokes equations, discretized with multi-patch Isogeometric Analysis. In the last years, several inf-sup stable discretizations for the Stokes problem have been proposed, often the analysis was restricted to single-patch domains. We focus on one of the simplest approaches, the isogeometric Taylor--Hood element. We show how stability results for single-patch domains can be carried over to multi-patch domains. While this is possible, the stability strongly depends on the shape of the geometry. We construct a Dual-Primal Isogeometric Tearing and Interconnecting (IETI-DP) solver that does not suffer from that effect. We give a convergence analysis and provide numerical tests.
\end{abstract}



\section{Introduction}
\label{sec:1}

Isogeometric Analysis (IgA) was introduced in Ref.~\refcite{hughes2005isogeometric} as a technique for discretizing partial differential equations (PDEs); see also Ref.~\refcite{Cottrell:Hughes:Bazilevs} and references therein. The original idea is to improve the integration of simulation and computer aided design (CAD), compared to the classical finite element (FEM) simulation. This is achieved by representing both the computational domain and the solution of the PDE as linear combination of tensor-product B-splines or non-uniform rational B-splines (NURBS).
Simple computational domains can be parameterized using a single geometry mapping. More complicated domains are usually composed of multiple patches, each parameterized with its own geometry mapping. Such domains are called multi-patch domains. We are interested in fast solvers for the Stokes system, discretized using IgA on multi-patch domains.

For the discretization of the Stokes equations, we need \textit{inf-sup} stable discretizations. Several inf-sup stable elements from the FEM world have been generalized for the IgA framework, like N{\'e}d{\'e}lec, Raviart--Thomas and Taylor--Hood elements, cf. Ref.~\refcite{bressan2013isogeometric,buffa2011isogeometric,evans2013isogeometric}. These methods have in common that the same grid is used both for the velocity and the pressure; an alternative approach based on different grids for velocity and pressure is the subgrid approach, cf. Ref.~\refcite{bressan2013isogeometric}. In this paper, we focus on the generalized Taylor--Hood element. A stability estimate was proven in Ref.~\refcite{bressan2013isogeometric} for tensor-product B-splines and it was later extended to hierarchical splines in Ref.~\refcite{bressan2018inf}. The analysis provides lower bounds for the inf-sup constant that is independent of the grid size. Numerical experiments suggest that the inf-sup constant is also independent of the spline degree; an analysis confirming this, is not known to the authors. In this paper, we use these estimates to show a corresponding stability estimate for multi-patch domains (see Theorem~\ref{theo:inf-supGlobal}).

Usually, the inf-sup constant for the discretized problem depends on the inf-sup constant for the continuous problem, which in turn depends on the domain. Similarly, our inf-sup constant depends on the patch-local inf-sup constants for the chosen discretization and, additionally, on the global inf-sup constant for the continuous problem. The inf-sup constant can be computed explicitly for simple domains: For a rectangle, the inf-sup constant behaves like the length of the larger side, divided by the length of the shorter side, this means that the stability degrades if the domain gets longer and thinner, cf. Ref.~\refcite{costabel2015inf}. We consider a multi-patch computational domain which suffers from this effect. A common strategy for preconditioning the Stokes problem is to use a block diagonal preconditioner with a stiffness matrix for the velocity and a mass matrix for the pressure. The efficiency of this approach depends on the inf-sup constant. Thus, it is inefficient for the domains we consider.

To remedy this, we are interested in a solver whose convergence behavior does not depend on the inf-sup constant for the global problem. We consider FETI-DP methods, which were originally introduced in Ref.~\refcite{farhat2001feti}. We decompose the original problem into patch-local problems, where we know that the inf-sup constant is reasonably large. For multi-patch IgA domains, FETI-DP methods are a natural choice as the patches can serve as substructures. FETI-DP was first adapted to IgA in Ref.~\refcite{kleiss2012ieti} and named the \textit{Dual-Primal Isogeometric Tearing and Interconnecting} (IETI-DP) method.
For second-order elliptic boundary value problems, IETI-DP methods have been extensively explored, see, e.g., Ref.~\refcite{hofer2017dual,hofer2019dual,SchneckenleitnerTakacs:2020,SchneckenleitnerTakacs:2021b}
and, e.g., Ref.~\refcite{BCPS:2013} for the similar BDDC method. In Ref.~\refcite{SchneckenleitnerTakacs:2020}, a convergence analysis is proven, which is, besides grid sizes and the patch diameters, also robust in the spline degree and spline smoothness.

The extension of these results to the Stokes equations poses several challenges. FETI-DP solvers for the Stokes problem have also been considered in the context of finite element (FEM) discretizations, see, e.g., Ref.~\refcite{kimleepark,li2005dual,tuli} for the case of two dimensions and Ref.~\refcite{tu2015feti} for the case of three dimensions. In the context of IgA, a FETI-DP like solver has been applied in a single-patch setting to a generalized Taylor-Hood element in Ref.~\refcite{pavarino2016isogeometric}. The substructures used for the solver are non-overlapping parts of the patch. Isogeometric discretizations distinguish themselves by the smoothness of the functions. The solver from Ref.~\refcite{pavarino2016isogeometric} preserves this smoothness also between the substructures. Moreover, the authors have proposed a solver for the elasticity problem for incompressible and almost incompressible materials. These results have recently been extended in Ref.~\refcite{widlund2021block}. 

We follow the philosophy of IETI-DP solvers, this means that we consider multi-patch domains and use the patches as substructures for the solver. We realize the coupling between the patches based on the minimum smoothness requirements that guarantee a conforming discretization. Since the velocity lives in the Sobolev space $H^1$, we impose continuity across the patches. For the pressure, which is an $L^2$ function, we do not realize any coupling between the patches. As for any IETI-DP method, we have to choose primal degrees of freedom: We use the corner values of each velocity component, the integral of the normal component of the velocity on each of the edges and patchwise averages of the pressure; a similar choice can be found in Ref.~\refcite{li2005dual}. This choice ensures that local system is non-singular. Although the Stokes system is indefinite, we can reduce the system by a Schur complement approach to a symmetric positive definite system formulation. The system is preconditioned with a scaled Dirichlet preconditioner, which is based on solving patch-local vector valued Poisson problems. 

We give a condition number bound for the Schur complement formulation of the IETI-DP solver, preconditioned with the scaled Dirichlet preconditioner
(see Theorem~\ref{thrm:fin}). This analysis uses many results that have been developed in Ref.~\refcite{SchneckenleitnerTakacs:2020} for the Poisson problem. The analysis is explicit with respect to grid sizes, the patch diameters, the spline degree, and the inf-sup constants for the local problems. Numerical experiments for the proposed method are provided, but we also refer to Ref.~\refcite{sogn2021dual}, where alternative choices of the primal degrees of freedom and alternative setups of the scaled Dirichlet preconditioner are considered.

The remainder of this paper is organized as follows. We present the model problem in Section~\ref{sec:2}. In Section~\ref{sec:3}, we introduce an inf-sup stable discretization for multi-patch domains and prove the stability. A IETI-DP solver is proposed in Section~\ref{sec:4}, which is analyzed in the subsequent Section~\ref{sec:5}. We conclude the main part of the paper with Section~\ref{sec:6}, where the results from numerical experiments are presented and analyzed. The Appendix contains some of the proofs.

\section{The model problem}
\label{sec:2}
As model problem, we consider the Stokes equations with
homogeneous Dirichlet boundary conditions in two dimensions. In
detail, the model problem is as follows.
Let $\Omega\subset\mathbb{R}^2$ be an open and bounded domain with Lipschitz boundary $\partial \Omega$.
$L^2(\Omega)$ and $H^s(\Omega)$ denote the standard Lebesgue and
Sobolev spaces on $\Omega$. Moreover, $L^2_0(\Omega)$ 
is the subspace of functions with a mean value of zero,
i.e., $L^2_0(\Omega):=\lbrace q\in L^2(\Omega)\,:\, (q,1)_{L^2(\Omega)} = 0\rbrace$,
and $H^1_0(\Omega)$ is the subspace of $H^1(\Omega)$ of functions 
with vanishing trace. For a given right-hand side $\mathbf{f} \in \left[L^2(\Omega)\right]^2$, find $(\mathbf{u},p)\in \left[H^1_0(\Omega)\right]^2\times L^2_0(\Omega)$ such that
\begin{align}
  \label{eq:stokes}
  \begin{split} 
	(\nabla \mathbf{u}, \nabla \mathbf{v})_{L^2(\Omega)} + (p, \nabla \cdot \mathbf{v})_{L^2(\Omega)} &= (\mathbf{f},\mathbf{v})_{L^2(\Omega)} \quad \foralls \mathbf{v} \in \left[H^1_0(\Omega)\right]^2,\\
	(\nabla \cdot \mathbf{u}, q)_{L^2(\Omega)} \qquad\qquad\qquad \quad\;\;&= 0 \qquad\qquad \quad \foralls q \in L^2_0(\Omega).
  \end{split} 
\end{align}

The existence and uniqueness of a solution to problem~\eqref{eq:stokes} is known for any domain $\Omega$ with Lipschitz boundary; for a proof, see, e.g., Ref.~\refcite{Bramble:2003}, for further information also
Ref.~\refcite{Necas:1967,fortin1991mixed} and references therein. The analysis
is based on Brezzi's theorem Ref.~\refcite{Brezzi:1974}, where one shows that there are constants $0<\alpha\le\gamma$ and $0<\beta\le\delta$ such that one has
\emph{coercivity}
\begin{equation}
  \label{eq:coercivity}
		(\nabla \mathbf u, \nabla \mathbf u)_{L^2(\Omega)}
		\ge \alpha
		\|\mathbf u\|_{H^1(\Omega)}^2 \quad \forall\, \mathbf u \in [H^1_0(\Omega)]^2,
\end{equation}
\emph{inf-sup stability}
\begin{equation}
  \label{eq:infsupcont}
		\sup_{\mathbf{u} \in [H^1_0(\Omega)]^2}
			\frac{ (\nabla \cdot \mathbf{u}, p)_{L^2(\Omega)} }{ \|\mathbf{u}\|_{H^1(\Omega)}  }
			\ge \beta \|p\|_{L^2(\Omega)}\quad \forall\, p \in L^2_0(\Omega)
\end{equation}
and \emph{boundedness}
\begin{equation}
  \label{eq:boundedness}
  	\begin{aligned}
		(\nabla \mathbf u, \nabla \mathbf v)_{L^2(\Omega)}
			&\le \gamma \|\mathbf u\|_{H^1(\Omega)}\|\mathbf v\|_{H^1(\Omega)}
			\quad \forall\, \mathbf u,\mathbf v \in [H^1(\Omega)]^2,\\
		(\nabla \cdot \mathbf{u}, p)_{L^2(\Omega)}
			&\le \delta \|\mathbf{u}\|_{H^1(\Omega)} \|p\|_{L^2(\Omega)}
			\quad \forall\, \mathbf{u} \in [H^1_0(\Omega)]^2,\; p \in L^2_0(\Omega).
	\end{aligned}
\end{equation}
 In~\eqref{eq:infsupcont}, we do not explicitly mention
that $\mathbf{u}\not=0$. Also formulas with suprema that follow are to
be understood in that way.

The only non-trivial condition is the inf-sup
stability~\eqref{eq:infsupcont}. Coercivity~\eqref{eq:coercivity} is a direct consequence of Friedrichs' inequality (cf., e.g., Lemma~1.31 in Ref.~\refcite{Pechstein:2013a}) and boundedness~\eqref{eq:boundedness} (with $\gamma=1$ and $\delta=\sqrt{d}=\sqrt{2}$) follows directly from the Cauchy-Schwarz inequality.

\section{Stable discretizations}
\label{sec:3}

In the following, we introduce a conforming discretization of the Stokes equations which again satisfies the conditions of Brezzi's theorem. Certainly,~\eqref{eq:coercivity} and~\eqref{eq:boundedness} carry directly over to conforming discretizations. The story is different for the inf-sup condition, which has to be verified for the discretized problem as well. For single-patch Isogeometric Analysis, such stable discretizations have been introduced previously. After introducing the representation of the computational domain in Subsection~\ref{subsec:3:1} and the standard concepts of isogeometric functions in Subsection~\ref{subsec:3:2}, we replicate the details of the isogeometric Taylor--Hood element, which we use in our further considerations, in Subsection~\ref{subsec:3:3}. In Subsection~\ref{subsec:3:4}, we discuss the extension of these results to multi-patch Isogeometric Analysis and the dependence of the inf-sup constant on the shape of the computational domain. In Subsection~\ref{subsec:3:5}, we present and discuss numerical results that illustrate the dependence of the stability on the shape of the geometry.

\subsection{Representation of the geometry}
\label{subsec:3:1}

We assume that the computational domain $\Omega\subset \mathbb R^2$ is composed
of $K$ non-overlapping patches $\Omega^{(k)}$, i.e., the domains
$\Omega^{(k)}$ are open and bounded domains with Lipschitz boundary such that 
\begin{align*}
	\overline{\Omega} = \bigcup_{k=1}^K \overline{\Omega^{(k)}} \quad \text{and}\quad
	\Omega^{(k)} \cap \Omega^{(\ell)} = \emptyset \quad\text{for all}\quad k \neq \ell,
\end{align*}
where $\overline{T}$ denotes the closure of the set $T$.
We need that that the patches form an admissible decomposition, i.e.,
that there are no T-junctions.
\begin{assumption}
	\label{ass:conforming}
	For any two patch indices $k\not=\ell$, the set	$\partial{\Omega^{(k)}} \cap \partial{\Omega^{(\ell)}}$	is either a common edge
    $\Gamma^{(k,\ell)}:=\partial{\Omega^{(k)}} \cap \partial{\Omega^{(\ell)}}$,
    a common vertex or empty.
\end{assumption}
This assumption is necessary to allow a fully matching discretization, which is a prerequisite for an $H^1$-conforming discretization.
Recently, a IETI solver for the Poisson equation was proposed that allows a decomposition including T-junctions, cf. Ref.~\refcite{SchneckenleitnerTakacs:2021b}. That approach uses a discontinuous Galerkin method in order to couple the patches. Since we focus on conforming discretizations,
we cannot use an analogous approach.

For any patch index $k$, the set $\mathcal{N}_\Gamma(k)$ contains the indices $\ell$ of patches $\Omega^{(\ell)}$ that share an edge with $\Omega^{(k)}$. The common vertices of two or more
patches -- that are not located on the (Dirichlet) boundary -- are denoted by
$x_1,\ldots,x_J$. For each $j=1,\ldots,J$,
the set $\mathcal{N}_x(j)$ contains the indices of all
patches $\Omega^{(k)}$ such that $x_j\in\partial\Omega^{(k)}$. We assume that the number
of patches sharing one vertex is uniformly bounded.
\begin{assumption}\label{ass:neighbors}
		There is a constant $C_2>0$ such that
		\[
				|\mathcal{N}_x(j)|\le C_2
				\quad\foralls j=1,\ldots,J.
		\]
\end{assumption}
Each patch $\Omega^{(k)}$ is parameterized by a geometry mapping 
\begin{align*}
	\mathbf{G}_k:\widehat{\Omega}:=(0,1)^2 \rightarrow \Omega^{(k)}:=\mathbf{G}_k(\widehat{\Omega}) \subset \mathbb{R}^2, 
\end{align*}
which can be continuously extended to the closure of the parameter domain
$\widehat{\Omega}$. In IgA, the geometry mapping is typically represented using B-splines or NURBS. As usual, the computational methods do not depend on such a representation. We only assume that the geometry mappings are not too much distorted, i.e., that the following assumption holds.
\begin{assumption}
	\label{ass:nabla}
	There is a
	constant $C_3>0$ such that
	\begin{align*}
		\| \nabla \mathbf G_k \|_{L^\infty(\widehat{\Omega})} \le C_3\, H_k
		\quad\text{and}\quad
		\| \nabla \mathbf G_k^{-1} \|_{L^\infty(\widehat{\Omega})} \le C_3\, \frac{1}{H_k},
	\end{align*}
	where $H_k$ is the diameter of the patch $\Omega^{(k)}$,
	holds for all $k=1,\ldots,K$.
\end{assumption}
We need one more assumption in order to analyze the inf-sup stability of the global problem. This is a condition that is specific for the analysis of the Stokes equations. The assumption guarantees that the interfaces are bent by uniformly less than 180 degrees. 
If Assumption~\ref{ass:nabla} holds, each subdivision of $\Omega$ into patches that does not satisfy
Assumption~\ref{ass:normals} can be converted into a subdivision satisfying this condition by (uniformly) subdividing the patches sufficiently often.
\begin{assumption}\label{ass:normals}
We assume that there is a constant $C_4>0$ such that,
on each interface $\Gamma^{(k,\ell)}=\Gamma^{(\ell,k)}$, there is some point
$\overline{x}^{(k,\ell)} = \overline{x}^{(\ell,k)} \in \Gamma^{(k,\ell)}$ with
\[
		 \mathbf n^{(k)}(\overline{x}^{(k,\ell)}) \cdot \mathbf n^{(k)}(x)  \ge C_4
		\quad \forall x\in \Gamma^{(k,\ell)},
\]
where $\mathbf n^{(k)}$ is the outer normal vector on $\Omega^{(k)}$.
\end{assumption}

\subsection{Isogeometric functions}
\label{subsec:3:2}

On the parameter domain $\widehat \Omega=(0,1)^2$, we choose a B-spline space, which depends on a freely chosen vector of breakpoints 
\[
		Z^{(k,\delta)}:= (\zeta_0^{(k,\delta)},\ldots,
				\zeta_{N^{(k,\delta)}}^{(k,\delta)})
				\quad\mbox{with}\quad
				0 = \zeta_0^{(k,\delta)} < \ldots <
				\zeta_{N^{(k,\delta)}}^{(k,\delta)} =1
\]
for each patch $k$ and each spacial direction $\delta\in\{1,2\}$, a 
freely chosen degree parameter $\degree\in\mathbb N:=\{1,2,3,\ldots\}$ and a freely chosen smoothness parameter $\smoothness\in\{0,1,\ldots,\degree-1\}$. 
Based on these vectors of breakpoints, we introduce spline spaces
of degree $\degree$ and smoothness $\smoothness$:
\[
		S^{(k,\delta,\degree,\smoothness)}:=\{ u \in C^{\smoothness}(0,1)
					:  u|_{(\zeta_{i-1}^{(k,\delta)},\zeta_{i}^{(k,\delta)})} \in \mathbb P^\degree
					\mbox{ for all } i=1,\ldots,N^{(k,\delta)}
					\},
\]
where $\mathbb P^\degree$ is the space of polynomials of degree $\degree$.
For each such set, we choose the basis that is obtained by the
Cox-de Boor formula (cf. (2.1) and (2.2) in Ref.~\refcite{Cottrell:Hughes:Bazilevs}); for the application of the Cox-de Boor formula, one uses
a knot vector obtained from the vector of breakpoints by repeating the first and
the last breakpoint $\degree+1$ times and by repeating all other breakpoints
$\degree-\smoothness$ times.

Based on these univariate splines, we introduce the corresponding
tensor-product spline space 
\begin{align*}
		S^{(k,\degree,\smoothness)}
		&:= S^{(k,1,\degree,\smoothness)} \otimes S^{(k,2,\degree,\smoothness)} \\
		&=\left\{u : u(x,y) =\sum_{n=1}^N v_n^{(1)}(x)v_n^{(2)}(y)
							\mbox{ with } v_n^{(\delta)}
							\in S^{(k,\delta,\degree,\smoothness)}
							\mbox{ for } N\in\mathbb N
							\right\}
\end{align*}
as discretization space on the parameter domain $\widehat\Omega$, and
equip it with the standard tensor-product basis.

The function spaces on the physical patches $\Omega^{(k)}$ are defined via
the \emph{pull-back principle}, so we define a space of functions
$\Omega^{(k)} \rightarrow \mathbb R$ via
\[
		V^{(k,\degree,\smoothness)}
		:=
		\{ v :
			v\circ \mathbf G_k \in S^{(k,\degree,\smoothness)}
		\}.
\]
The grid size $\widehat h_k$ on the parameter domain and the grid size $h_k$ on the physical patch are defined by
\[
		\widehat h_k:=\max\{ \zeta_i^{(k,\delta)}-\zeta_{i-1}^{(k,\delta)}
								\,:\, i=1,\ldots,N^{(k,\delta)}  ,\, \delta=1,2 \}
								\quad\mbox{and}\quad
								h_k:=H_k \widehat h_k,
\]
where the definition of the latter is motivated by Assumption~\ref{ass:nabla}.
We assume that the grids are quasi-uniform. 
\begin{assumption}\label{ass:quasiuniform}
	There is a constant $C_5$ such that for $k=1,\ldots,K$
  \[
  C_5 \widehat h_k\leq \widehat h_{k,\min}
  := \min\{ \zeta_i^{(k,\delta)}-\zeta_{i-1}^{(k,\delta)}
								\,:\, i=1,\ldots,N^{(k,\delta)}  ,\, \delta=1,2 \}.
  \]
\end{assumption}
Note that Assumption~\ref{ass:nabla} allows us to relate the norm of the function
on the physical patch and the corresponding function on the parameter domain. There is
a constant $c_G>0$, only depending on the constant from Assumption~\ref{ass:nabla}, 
such that 
\begin{equation}\label{eq:geoequiv}
\begin{aligned}
	c_G^{-1} | v \circ \mathbf G_k |_{H^1(\widehat\Omega)}^2
	&\le | v |_{H^1(\Omega^{(k)})}^2
	\le c_G | v \circ \mathbf G_k |_{H^1(\widehat\Omega)}^2
	&\hspace{1em}\foralls v\in H^1(\Omega^{(k)}),\\
	c_G^{-1} H_k^2 \| v \circ \mathbf G_k \|_{L^2(\widehat\Omega)}^2
	&\le \| v \|_{L^2(\Omega^{(k)})}^2
	\le c_G H_k^2 \| v \circ \mathbf G_k \|_{L^2(\widehat\Omega)}^2
	&\hspace{1em}\foralls v\in L^2(\Omega^{(k)}).
\end{aligned}
\end{equation}
Using a standard Poincaré
inequality (cf., e.g., Lemma~1.27 in Ref.~\refcite{Pechstein:2013a}), we obtain
\begin{equation}\label{eq:poincare}
\begin{aligned}
		\inf_{c\in\mathbb R} \|u-c\|_{L^2(\Omega^{(k)})}
			& \le c_G^{1/2} H_k \inf_{c\in\mathbb R} \|u\circ \mathbf G_k-c\|_{L^2(\widehat\Omega)}
				\le c_G^{1/2} \widehat c_P H_k |u\circ \mathbf G_k|_{H^1(\widehat\Omega)}\\
				&\le c_G \widehat c_P H_k |u|_{H^1(\Omega^{(k)})}
				\qquad \foralls	u\in H^1_0(\Omega)
				,
\end{aligned}
\end{equation}
where $\widehat c_P$ is the Poincaré constant for the parameter
domain $\widehat\Omega=(0,1)^2$. This means that the Poincaré constant for $\Omega^{(k)}$
only depends on $c_G$ and $H_k$. A completely analogous result for
the Friedrichs' inequality is straight forward: For
all patches $\Omega^{(k)}$, where at least one edge is located on the (Dirichlet)
boundary, we have using a standard Friedrichs' inequality (cf., e.g., Lemma~1.31 in Ref.~\refcite{Pechstein:2013a})
\begin{equation}\label{eq:friedrichs}
\begin{aligned}
		\|u\|_{L^2(\Omega^{(k)})}
			\le c_G \widehat c_F H_k |u|_{H^1(\Omega^{(k)})}
				\qquad \foralls
				u\in H^1_0(\Omega)
				,
\end{aligned}
\end{equation}
where $\widehat c_F$ is the Friedrichs' constant for the parameter
domain~$\widehat\Omega$.

\subsection{Stable discretizations for the single-patch case}
\label{subsec:3:3}

As discretization space for the single-patch case, we use the 
isogeometric Taylor--Hood element, as proposed in Ref.~\refcite{bressan2013isogeometric}.
It uses the same grid for all velocity components and for the pressure and
can be defined based on any underlying spline degree parameter $\degree\in \mathbb{N}$ and any underlying smoothness $\smoothness\in\{0,\ldots,\degree-1\}$.

The idea of the isogeometric Taylor--Hood element is to use splines of degree $\degree+1$
and smoothness $\smoothness$, which vanish on the (Dirichlet) boundary, for the velocity and
splines of degree $\degree$ and smoothness $\smoothness$ with vanishing mean value for
the pressure. Our approach is to use these spaces for each of the patches $\Omega^{(k)}$,
however we have to modify the spaces accordingly. So, Dirichlet boundary conditions are not
to be imposed on $\partial\Omega^{(k)}$, but only on $\Gamma_D^{(k)}:=\partial\Omega^{(k)}\cap
\partial\Omega$. Analogously, the condition on the mean value of the pressure only holds
for the whole domain $\Omega$.

So, we define as follows. The function spaces for the parameter domain $\widehat{\Omega}$ are
\begin{align*}
  \mathbf{\widehat{V}}^{(k)} := \left\lbrace \mathbf{v}\in [S^{(k,\degree+1,\smoothness)}]^2\, : \, \mathbf{v}|_{\widehat{\Gamma}^{(k)}_D} = 0  \right\rbrace \quad \text{and}\quad   \widehat{Q}^{(k)} = S^{(k,\degree,\smoothness)},
\end{align*}
where $\mathbf{u}|_{\widehat{\Gamma}^{(k)}_D}$ is the restriction of $\mathbf{u}$ to $\widehat{\Gamma}^{(k)}_D:=\textbf G_k^{-1}(\Gamma^{(k)}_D)$,
the pre-image of the Dirichlet boundary portion $\Gamma^{(k)}_D$. On the physical patch $\Omega^{(k)}$, the spaces are defined through the pull back principle:
\begin{align*}
  \mathbf{V}^{(k)} = \mathbf{\widehat{V}}^{(k)} \circ \mathbf{G}^{-1}_k \quad \text{and} \quad Q^{(k)} = \widehat{Q}^{(k)} \circ \mathbf{G}^{-1}_k.
\end{align*}
As basis for the space $\mathbf{\widehat{V}}^{(k)}$, we choose the basis functions of the standard tensor-product B-spline basis that vanish on $\widehat{\Gamma}^{(k)}_D$. Their images under the geometry function $\mathbf G_k$ form the basis for the space $\mathbf{V}^{(k)}$. The bases for $\widehat Q^{(k)}$ and $Q^{(k)}$ are defined analogously. 

In Ref.~\refcite{bressan2013isogeometric}, it was shown that the isogeometric Taylor--Hood element
is inf-sup stable. Certainly, this only holds if we have boundary conditions on all of
$\partial\Omega^{(k)}$ and the averaging condition locally, i.e., we have
\begin{equation}
  \label{eq:infsuplocal}
		\sup_{\mathbf{u} \in \mathbf{V}^{(k)}\cap [H^1_0(\Omega^{(k)})]^2}
			\frac{ (\nabla \cdot \mathbf{u}, p)_{L^2(\Omega^{(k)})} }{ |\mathbf{u}|_{H^1(\Omega^{(k)})}  }
			\ge \beta_{k}\|p\|_{L^2(\Omega^{(k)})}\quad \forall\, p \in Q^{(k)}\cap L^2_0(\Omega^{(k)}),
  \end{equation}
where the inf-sup constant $\beta_{k}$ is independent of the grid size $h_k$, but it depends
on $\mathbf{G}_{k}$ and the constant from Assumption~\ref{ass:quasiuniform}. Since the discretization is conforming, coercivity~\eqref{eq:coercivity} and boundedness~\eqref{eq:boundedness} are also satisfied for the discretion problem. 
\begin{remark}\label{remark:probust}
  Extensive numerical experiments indicate that the constant $\beta_k$ is independent of the spline degree $\degree$. However, at the time of writing, no such proof is known to the authors.
\end{remark}

\subsection{Stable discretization in the multi-patch case}
\label{subsec:3:4}

In this section, we introduce the global function spaces $\mathbf V \subset [H^1_0(\Omega)]^2$
and $Q \subset L^2_0(\Omega)$. Since we set up a conforming discretization, we need that
the space $\mathbf V$ is continuous. To be able to set up a continuous
global function space, we need that the discretization is fully matching,
i.e., that the following assumption holds.
\begin{assumption}\label{assumption:conforming}
  For every interface $\Gamma^{(k,\ell)}$ between two patches,
  the following statement holds true.
  For any basis function in the basis for $\mathbf V^{(k)}$ having
  support on $\Gamma^{(k,\ell)}$, there is exactly one basis function
  in the basis for $\mathbf V^{(k)}$
  such that they agree on the interface $\Gamma^{(k,\ell)}$.
\end{assumption}
This assumption holds if the spline degree, the vector of breakpoints and the geometry mapping agree on all common interfaces.
For each of the matching basis functions in Assumption~\ref{assumption:conforming}, we set the corresponding coefficients to have the same value.
In this way, we obtain an $H^1$-conforming discretization space. Note, this is not done for the pressure space since it only needs to be $L^2$-conforming. We can now state the overall discretization space. For the velocity, we use
\[
\mathbf{V} = \left\lbrace \mathbf{v}\in [H^1_0 (\Omega)]^2 \,:\, \mathbf{v}|_{\Omega^{(k)}}\in \mathbf{V}^{(k)} \text{ for } k=1,\ldots,K\right\rbrace
\]
and for the pressure, we use 
\[
Q =\left\lbrace q\in L^2_0 (\Omega) \,:\, q|_{\Omega^{(k)}}\in Q^{(k)} \text{ for } k=1,\ldots,K\right\rbrace.
\]
The discretized Stokes problem reads as follows. Find $(\mathbf{u},p)\in \mathbf V \times Q$ such that
\begin{align}
  \label{eq:discstokes}
  \begin{split} 
	(\nabla \mathbf{u}, \nabla \mathbf{v})_{L^2(\Omega)} + (p, \nabla \cdot \mathbf{v})_{L^2(\Omega)} &= (\mathbf{f},\mathbf{v})_{L^2(\Omega)} \quad \foralls \mathbf{v} \in \mathbf V,\\
	(\nabla \cdot \mathbf{u}, q)_{L^2(\Omega)} \qquad\qquad\qquad \quad\;\;&= 0 \qquad\qquad \quad \foralls q \in Q.
  \end{split} 
\end{align}

Now, we prove an inf-sup stability result for the multi-patch
case, which uses the inf-sup stability of the continuous problem, i.e.,~\eqref{eq:infsupcont}, and the inf-sup stability result
for the single-patch case, i.e.,~\eqref{eq:infsuplocal}.

Before we can prove the main inf-sup result, we need some auxiliary results. Note that $Q$
is the direct sum of
\[
  Q_0 := \{ q_0 \in Q \;:\; q_0|_{\Omega^{(k)}}\in Q^{(k)}\cap L^2_0(\Omega^{(k)})\mbox{ for } k=1,\ldots,K \},
\]
the space of function with zero average on each patch, and
\[
	Q_1 := \{q_1\in L_0^2(\Omega) \;:\; q_1|_{\Omega^{(k)}} \mbox{ is constant for } k=1,\ldots,K\},
\]
the space of patchwise constant functions. First, we state the existence of a Fortin operator.
\begin{lemma}
  \label{lamma:FortinOp}
		There exists an operator $\mathbf{\Pi}_F: [H^1_0(\Omega)]^2\rightarrow \mathbf{V}$ such that
		\begin{equation}\label{eq:pi:stabtmp}
			|\mathbf{\Pi}_F \mathbf{u}|_{H^1(\Omega)} \le c_F
			 |\mathbf{u}|_{H^1(\Omega)} 
			 \quad \forall\, \mathbf{u}\in [H^1_0(\Omega)]^2
		\end{equation}
		and
		\begin{equation}\label{eq:pi:divpres0}
			(\nabla \cdot (I-\mathbf{\Pi}_F)\mathbf{u}, p_1)_{L^2(\Omega)}=0	
			\quad \forall\,	(\mathbf{u},p_1)\in [H^1_0(\Omega)]^2 \times Q_1,
		\end{equation}
		where $c_F\ge1$ is a constant that
		only depends on the constants from the
		Assumptions~\ref{ass:neighbors}, \ref{ass:nabla} and \ref{ass:normals}.
\end{lemma}
The proof of this lemma is given in the Appendix.
We now show an inf-sup estimate for the pressure space of patchwise constants.
\begin{lemma}
  \label{lemma:infsupconst}
We have the inf-sup estimate
  \begin{equation}
  \label{eq:infsupglobalconst}
		\sup_{\mathbf{u} \in \mathbf{V}}
			\frac{ (\nabla \cdot \mathbf{u}, p_1)_{L^2(\Omega)} }{ |\mathbf{u}|_{H^1(\Omega)}  }
			\ge \frac{\beta}{c_F}\|p_1\|_{L^2(\Omega)}\quad \forall\, p_1 \in Q_1,
  \end{equation}
  where $\beta$ is as in~\eqref{eq:infsupcont} and
  $c_F$ is as in Lemma~\ref{lamma:FortinOp}.
\end{lemma}
\begin{proof}
  Let $p_1 \in Q_1$ be arbitrary but fixed.
  From the continuous inf-sup condition~\eqref{eq:infsupcont}, it follows that 
  there exists a $\mathbf{v}\in\mathbf [H^1_0(\Omega)]^2$ such that
  \begin{equation}\label{eq:lemma:infsupconst}
  \frac{ (\nabla \cdot \mathbf{v}, p_1)_{L^2(\Omega)} }{ |\mathbf{v}|_{H^1(\Omega)} } \geq \beta\| p_1\|_{L^2(\Omega)}.
  \end{equation}
  By setting $\mathbf{u} := \mathbf{\Pi}_F \mathbf{v}$ and
  using Lemma~\ref{lamma:FortinOp} and~\eqref{eq:lemma:infsupconst}, we get
  \begin{align*}
  \sup_{\mathbf{u} \in \mathbf{V}}\frac{ (\nabla \cdot \mathbf{u}, p_1)_{L^2(\Omega)} }{ |\mathbf{u}|_{H^1(\Omega)}  }
  &\ge  \frac{ (\nabla \cdot \mathbf{\Pi}_F \mathbf{v}, p_1)_{L^2(\Omega)} }{ |\mathbf{\Pi}_F \mathbf{v}|_{H^1(\Omega)}}
  \ge \frac{1}{c_F} \frac{ (\nabla \cdot \mathbf{v}, p_1)_{L^2(\Omega)} }{ |\mathbf{v}|_{H^1(\Omega)}}\geq \frac{\beta}{c_F} \|p_1\|_{L^2(\Omega)},
  \end{align*}
  which finishes the proof.
\end{proof}

Using the inf-sup result above and the patchwise inf-sup result \eqref{eq:infsuplocal}, we can show a global discrete inf-sup result.
\begin{theorem}\label{theo:inf-supGlobal}
  Let $\mathbf{V}\times Q$ be the generalized Taylor--Hood space as defined in this Section. We have the inf-sup result
  \begin{equation}
      \label{eq:infsupGlobal}
	  \sup_{\mathbf{u} \in \mathbf{V}} \frac{(\nabla \cdot \mathbf{u}, p)_{L^2(\Omega)}}{|\mathbf{u}|_{H^1(\Omega)}}
	  \ge 
	  \underbrace{\frac{\beta\;\min_k\beta_k}{3c_F\delta}}_{\displaystyle \beta_h:=}
	  \|p\|_{L^2(\Omega)}\quad \forall\, p \in Q,
  \end{equation}
  where $\beta$ is as in~\eqref{eq:infsupcont}, $\beta_k$ as in~\eqref{eq:infsuplocal}, $\delta$ as in~\eqref{eq:boundedness} and
  $c_F$ as in Lemma~\ref{lamma:FortinOp}.
\end{theorem}
Before we prove this theorem, we give some remarks.
\begin{remark}\label{remark:2}
$\beta_h$ is independent of the grid sizes $h_k$ since the local
inf-sup constants $\beta_k$ and the Fortin constant $c_F$ are independent of $h_k$ and the inf-sup constant $\beta$ for the continuous problem is inherently independent of the discretization.
Robustness in the spline degree is obtained if the local inf-sup
constants $\beta_k$ are robust in the spline degree, which is an
unproven conjecture for the isogeometric Taylor--Hood element, see Remark~\ref{remark:probust}.
\end{remark}
\begin{remark}\label{remark:3}
We observe that the local inf-sup constants $\beta_k$ and the Fortin constant $c_F$ are independent of the shape of the overall
domain $\Omega$ (both only depend on the parameterization and thus
also of the shape of the individual patches and the
maximum number of patches meeting in one vertex), so any shape-dependence
observed in numerical results is due to the shape-dependence of 
$\beta$, the inf-sup constant of the continuous problem 
and thus inherent to the Stokes equations themselves.
\end{remark}
\begin{remark}\label{remark:4}
Finally, we observe that Theorem~\ref{theo:inf-supGlobal} is not restricted to the isogeometric
Taylor--Hood element. So, the proofs can be applied to any conforming, locally inf-sup stable pair of
discretization space, where the velocity space is fully matching at the interfaces and where all the bi-quadratic functions are in the underlying
space $\widehat{\mathbf V}^{(k)}$.
\end{remark}
\begin{proofof}{of Theorem~\ref{theo:inf-supGlobal}}
  Let $p\in Q$ be arbitrary but fixed and let $p=p_0+p_1$ with $p_0\in Q_0$ and $p_1\in Q_1$.
  From~\eqref{eq:infsuplocal}, we know that there are non-zero functions
  $\mathbf{u}^{(k)} \in \mathbf{V}^{(k)}\cap [H^1_0(\Omega^{(k)})]^2$ such that
  \[
  (\nabla \cdot \mathbf{u}^{(k)}, p_0)_{L^2(\Omega^{(k)})} \ge \widehat \beta |\mathbf{u}^{(k)}|_{H^1(\Omega^{(k)})} \|p_0\|_{L^2(\Omega^{(k)})},
  \quad \text{where} \quad \widehat \beta := \min_k \beta_k.
  \] 		
  The suprema in the inf-sup conditions are scaling invariant, so we can restrict ourselves to the choice
  $|\mathbf{u}^{(k)}|_{H^1(\Omega^{(k)})} = \|p_0\|_{L^2(\Omega^{(k)})}$.
  We define
  $\textbf u\in \mathbf{V}$ such that
  \[
  \mathbf{u}|_{\Omega^{(k)}}:= \mathbf{u}^{(k)}.
  \]
  Note that $\mathbf{u}$ vanishes on the interfaces between the patches.
  Summing up, we obtain
  \begin{equation}\label{eq:theo:inf-supGlobal}
  \begin{aligned}
	&(\nabla \cdot \mathbf{u}, p_0)_{L^2(\Omega)}    
	= \sum_{k=1}^K (\nabla \cdot \mathbf{u}^{(k)}, p_0)_{L^2(\Omega^{(k)})} 
	\ge \widehat \beta
	\sum_{k=1}^K |\mathbf{u}^{(k)}|_{H^1(\Omega^{(k)})}
	\|p_0\|_{L^2(\Omega^{(k)})} \\
	&\quad= \widehat \beta
	\sum_{k=1}^K \|p_0\|_{L^2(\Omega^{(k)})}^2 
	= \widehat \beta \left(\sum_{k=1}^K |\mathbf{u}^{(k)}|_{H^1(\Omega^{(k)})}^2\right)^{1/2}
	\left(\sum_{k=1}^K \|p_0\|_{L^2(\Omega^{(k)})}^2\right)^{1/2} \\
	&\quad= \widehat \beta
	|\mathbf{u}|_{H_1(\Omega)}
	\|p_0\|_{L^2(\Omega)}.
  \end{aligned}
  \end{equation}
  By applying integration-by-parts patchwise and using $\nabla p_1=0$
  and $\mathbf u|_{\partial\Omega^{(k)}}=0$, we obtain that
  \[
  (\nabla \cdot \mathbf{u}, p_1)_{L^2(\Omega)}	= -\sum^K_{k=1} (\mathbf{u}, \nabla p_1)_{L^2(\Omega^{(k)})}
  + \sum^K_{k=1} (\mathbf{u}\cdot \mathbf{n}, p_1)_{L^2(\partial \Omega^{(k)})}=0.
  \]
  By combining this with~\eqref{eq:theo:inf-supGlobal}, we obtain
  \begin{equation}\label{proof:1}
	\sup_{\mathbf{u} \in \mathbf{V}} \frac{ (\nabla \cdot \mathbf{u},  p)_{L^2(\Omega)} }{ |\mathbf{u}|_{H^1(\Omega)}} \ge \widehat \beta \|p_0\|_{L^2(\Omega)}.
  \end{equation}
  Lemma \ref{lemma:infsupconst} gives together with boundedness~\eqref{eq:boundedness}
  \begin{align*}\nonumber
	&\sup_{\mathbf{u} \in \mathbf{V}} \frac{ (\nabla \cdot \mathbf{u}, p)_{L^2(\Omega)}}{ |\mathbf{u}|_{H^1(\Omega)}}
	=\sup_{\mathbf{u} \in \mathbf{V}} \frac{ (\nabla \cdot \mathbf{u}, p_1)_{L^2(\Omega)} + (\nabla \cdot \mathbf{u}, p_0)_{L^2(\Omega)}}{ |\mathbf{u}|_{H^1(\Omega)}  }\\
	&\quad\ge \sup_{\mathbf{u} \in \mathbf{V}} \frac{ (\nabla \cdot \mathbf{u}, p_1)_{L^2(\Omega)} - \delta |\mathbf{u}|_{H^1(\Omega)} \| p_0\|_{L^2(\Omega)}}{ |\mathbf{u}|_{H^1(\Omega)}  }
	\ge \frac{\beta}{c_F}  \|p_1\|_{L^2(\Omega)} - \delta \|p_0\|_{L^2(\Omega)}.
  \end{align*}
  Using the triangle inequality, we have further
	\begin{equation}\label{proof:2}
	\sup_{\mathbf{u} \in \mathbf{V}} \frac{ (\nabla \cdot \mathbf{u}, p)_{L^2(\Omega)}}{ |\mathbf{u}|_{H^1(\Omega)}}
	\ge \frac{\beta}{c_F}  \|p\|_{L^2(\Omega)}
	- \left(\delta+\frac{\beta}{c_F}\right)
	\|p_0\|_{L^2(\Omega)}.
  \end{equation}
  By adding $(\delta+c_F^{-1}\beta)$ times
  inequality~\eqref{proof:1} and $\widehat \beta$ times
  inequality~\eqref{proof:2}, we finally get
  \[
  \sup_{\mathbf{u} \in \mathbf{V}} \frac{(\nabla \cdot \mathbf{u}, p)_{L^2(\Omega)}}{|\mathbf{u}|_{H^1(\Omega)}}
  \ge 
  \frac{\beta\widehat{\beta}}{c_F(\delta+c_F^{-1}\beta+\widehat{\beta})}\|p\|_{L^2(\Omega)}
  \ge
   \frac{\beta\widehat{\beta}}{3c_F\delta}\|p\|_{L^2(\Omega)},
  \]
  where we make use of $c_F\ge1$ and $\beta,\widehat{\beta}\le\delta$.
\end{proofof}

\subsection{Numerical exploration of the inf-sup constant}
\label{subsec:3:5}

\begin{figure}[htb]
	\centering
	\includegraphics[height=16em]{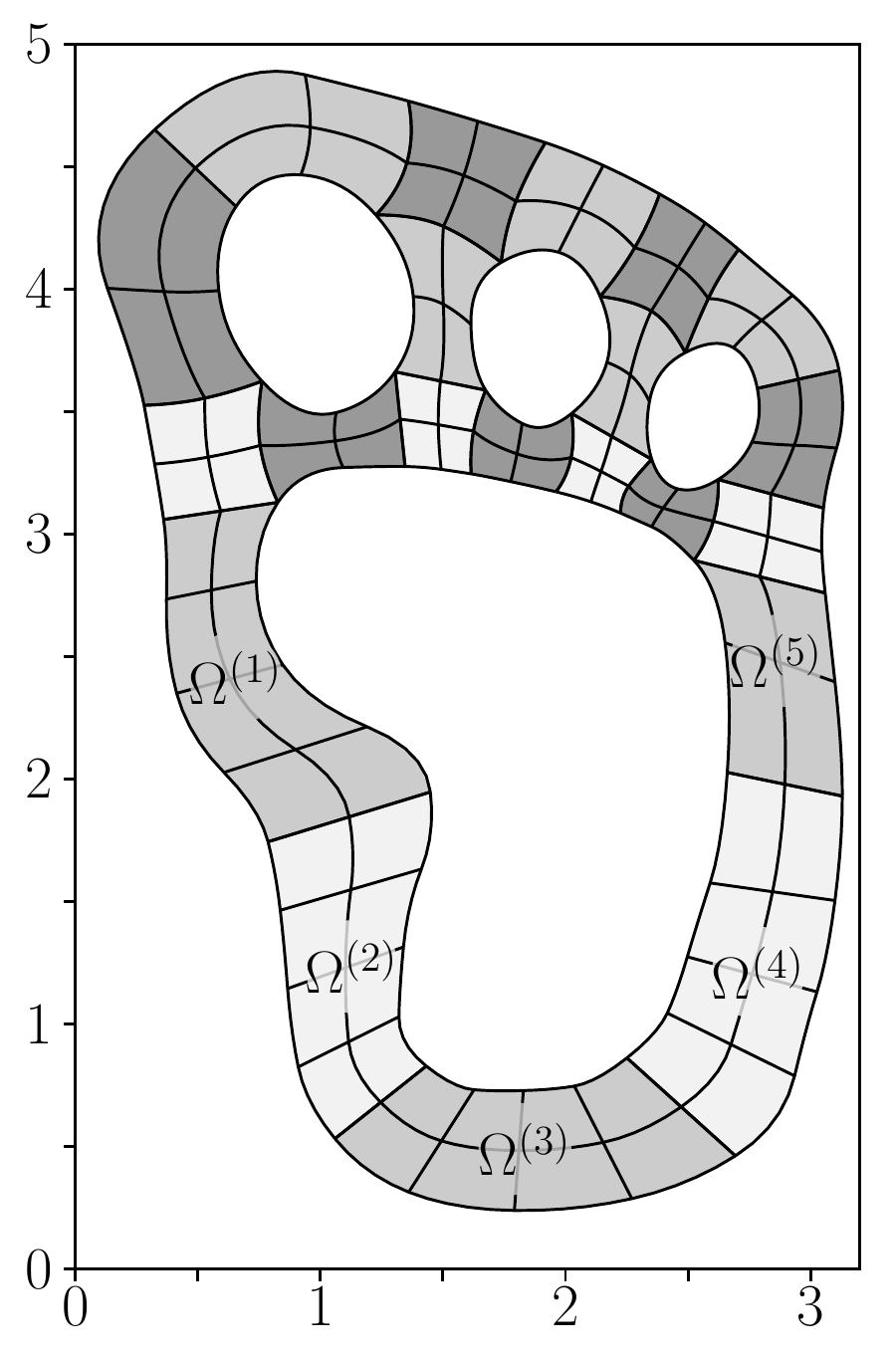}
	\caption{Multi-patch domain of a Yeti-footprint}
	\label{fig:yetiinfsup}
\end{figure}

In this subsection, we present numerical experiments that illustrate the dependence of the discrete inf-sup constant $\beta_h$ on the grid size, the spline degree and the shape of the geometry. This is done by deriving the condition number $\kappa$ of the (negative) Schur complement, preconditioned with the inverse of the mass matrix for the pressure. The relation between this condition number and the inf-sup constant is
\[
	\kappa = \frac{\sqrt{\delta_h}}{\sqrt{\beta_h}} \approx \frac{1}{\sqrt{\beta_h}},
\]
where $\delta_h\approx 1$ is the discrete version of the boundedness constant $\delta$ from \eqref{eq:boundedness}. As computational domain, we consider the Yeti-footprint, consisting of 21 patches, and domains obtained by combining a few of the patches. The Yeti-footprint is depicted in Figure~\ref{fig:yetiinfsup}; the patches are represented by different colors. The lines within each patch represent the coarsest ($\ell = 0$) grid on each patch. We obtain finer grids by performing $\ell$ uniform refinement steps ($\ell=0,1,2,3$) and test for various choices of the spline degree parameters $\degree$ ($\degree=1,2,\ldots,5$). The computed condition numbers are displayed in Table~\ref{tab:infsupcondition}. The left table shows the results for the single-patch domain $\Omega^{(1)}$ and the right table shows the results for the full Yeti-footprint $\overline{\Omega^{(1)}} \cup \cdots\cup \overline{\Omega^{(21)}}$. Due to the size the full Yeti-footprint, we could not calculate the condition numbers for $\ell = 3$, so these are left out.
As we see from the tables, the constants depend neither on the grid size (this is predicted by the theory), nor on the spline degree (cf. Remark~\ref{remark:probust}). We notice that the condition numbers are significantly larger for the full domain, which is a property of the geometry (cf. Remark~\ref{remark:3}).

\begin{table}[btp]
  \begin{center} 
  \begin{tabular}{|c|c|c|c|c|c|}
    \hline
    $\ell\setminus \degree$ & {$1$} & {$2$} & {$3$} & {$4$}& {$5$}  \\
    \hline
    \hline
    $0$  & $\!17.1\!$ & $\!17.5\!$ & $\!17.6\!$ & $\!17.6\!$ & $\!17.7\!$ \\ \hline
	$1$  & $\!17.6\!$ & $\!17.6\!$ & $\!17.7\!$ & $\!17.7\!$ & $\!17.7\!$ \\ \hline
	$2$  & $\!17.7\!$ & $\!17.7\!$ & $\!17.7\!$ & $\!17.7\!$ & $\!17.7\!$ \\ \hline
	$3$  & $\!17.7\!$ & $\!17.7\!$ & $\!17.7\!$ & $\!17.7\!$ & $\!17.7\!$ \\
    \hline
  \end{tabular}
 $\;\;$
    \begin{tabular}{|c|c|c|c|c|c|}
    \hline
    $\ell\setminus \degree$ & {$1$} & {$2$} & {$3$} & {$4$} & {$5$} \\
    \hline
    \hline
    $0$  & $\!243\!$ & $\!244\!$ & $\!243\!$ & $\!243\!$ & $\!243\!$ \\ \hline
	$1$  & $\!244\!$ & $\!244\!$ & $\!243\!$ & $\!243\!$ & $\!243\!$ \\ \hline
	$2$  & $\!244\!$ & $\!244\!$ & $\!243\!$ & $\!243\!$ & $\!243\!$ \\ \hline
    $3$  & $-$ & $-$ & $-$ & $-$ & $-$ \\ 
    \hline
  \end{tabular}
    \caption{Condition numbers for $\Omega^{(1)}$ (left) and the whole Yeti-footprint (right)}
    \label{tab:infsupcondition}
  \end{center} 
\end{table}

To explore the dependence on the shape of the domain further, we calculate condition numbers for partial Yeti-footprints $\overline{\Omega^{(1)}} \cup \cdots\cup \overline{\Omega^{(K)}}$ with
$K=1,2,3,4,5,21$. The corresponding condition numbers are computed for $\ell =2$ and $\degree = 2$
and presented in Table~\ref{tab:infsuppartialyeti}. We observe that the condition numbers increase steadily from $K=1$ to $K=5$, which is expected as the domains are similar to elongating rectangles. Note that the condition number does not increase too significantly from $K=5$ to $K=21$.
\begin{table}[btp]
  \begin{center} 
  \begin{tabular}{|c|c|c|c|c|c|c|}
    \hline
    Number of patches $K$
    & {\quad$1$\quad} & {\quad$2$\quad} & {\quad$3$\quad} & {\quad$4$\quad} & {\quad$5$\quad} & {\quad$21$\quad} \\
    \hline
    Condition number $\kappa$  & $17.7$ & $31.1$ & $79.5$ & $148$ & $184$ & $244$  \\ \hline    
  \end{tabular}
    \caption{Condition numbers for partial Yeti-footprints}
    \label{tab:infsuppartialyeti}
  \end{center} 
\end{table}

\section{A IETI-DP solver for the Stokes system}
\label{sec:4}

In this section, we outline the setup of the proposed IETI-DP method for solving the discretized Stokes problem~\eqref{eq:discstokes}.
For the IETI-DP method, we have to assemble the variational problem
locally. So, the still uncoupled problem is to find $(\kbox{\mathbf{u}},\kbox{p})\in \kbox{\mathbf V} \kbox{\times Q}$ such that
\begin{align}
  \nonumber
  \begin{split} 
	(\nabla \kbox{\mathbf{u}}, \nabla \kbox{\mathbf{v}})_{L^2(\Omega^{(k)})} + (\kbox{p}, \nabla \cdot \kbox{\mathbf{v}})_{L^2(\Omega^{(k)})} &\stackrel{!}{=} (\mathbf{f},\kbox{\mathbf{v}})_{L^2(\Omega^{(k)})} \; \foralls \kbox{\mathbf{v}} \in \kbox{\mathbf V},\\
	(\nabla \cdot \kbox{\mathbf{u}}, \kbox{q})_{L^2(\Omega^{(k)})} \qquad\qquad\qquad\qquad\qquad\;\;\,&\stackrel{!}{=} 0 \qquad\qquad\qquad\;\, \foralls \kbox{q} \in \kbox{Q}.
  \end{split} 
\end{align}
Here, we use the notation $\stackrel{!}{=}$ to remind ourselves that the coupling is still missing.
By discretizing these bilinear forms using the tensor-product bases,
we obtain linear systems
\begin{equation}
  \nonumber
    \kbox{A} \kbox{\underline{\mathbf{x}}} :=
		\begin{pmatrix}
		  \kbox{K}  & \kbox{D}^{\top} \\
		  \kbox{D}  & 0 
		\end{pmatrix}
        \begin{pmatrix}
          \underline{\mathbf{u}}^{(k)}\\
          \underline{p}^{(k)}
        \end{pmatrix}
        \stackrel{!}{=}
        \begin{pmatrix}
          \underline{\mathbf{f}}^{(k)} \\
          0
        \end{pmatrix}
        =:
        \underline{\mathbf{b}}^{(k)}.
\end{equation}
Here and in what follows, underlined quantities refer to the coefficient
representations of the corresponding functions. We
first represent the spaces $\mathbf{V}^{(k)}$ as a direct sum
\begin{equation}\label{eq:Vsplit}
	\mathbf{V}^{(k)} = \mathbf{V}^{(k)}_{\Gamma} \oplus \mathbf{V}^{(k)}_{\I},
\end{equation}
where $\mathbf{V}^{(k)}_{\I}$ is spanned by the basis functions which vanish on the interfaces and $\mathbf{V}^{(k)}_{\Gamma}$ is spanned by the remaining functions, i.e., the functions that are active on the interfaces (which includes the primal degrees of freedom). Assuming a corresponding ordering of the basis functions, we have
\begin{equation}
  \label{eq:patchStokeMatrix}
    \kbox{A} \kbox{\underline{\mathbf{x}}} =
		\begin{pmatrix}
		  \kbox{K_{\Gamma\Gamma}} & \kbox{K_{\Gamma\I}} & \kbox{D_\Gamma}^{\top} \\
		  \kbox{K_{\I\Gamma}}     & \kbox{K_{\I\I}}     & \kbox{D_\I}^{\top} \\
		  \kbox{D_\Gamma}         & \kbox{D_\I}         & 0 
		\end{pmatrix}
        \begin{pmatrix}
          \underline{\mathbf{u}}_\Gamma^{(k)}\\
          \underline{\mathbf{u}}_\I^{(k)}\\
          \underline{p}^{(k)}
        \end{pmatrix}
        \stackrel{!}{=}
        \begin{pmatrix}
          \underline{\mathbf{f}}_\Gamma^{(k)} \\
          \underline{\mathbf{f}}_\I^{(k)} \\
          0
        \end{pmatrix}
        =
        \underline{\mathbf{b}}^{(k)}.
\end{equation}
Analogous to the case of the Poisson problem, the local systems correspond to pure Neumann problems, unless the corresponding patch contributes to the Dirichlet boundary. These systems are not uniquely solvable since the constant velocities are in the null space of $\kbox{A}$.
To ensure that the system matrices of the patch-local problems are non-singular, we introduce primal degrees of freedom, whose continuity across the patches is enforced strongly (\emph{continuity conditions}):
\begin{itemize}
	\item the function values of the velocity at each of the corners of the patch, i.e.,
	\begin{equation}\label{eq:cornerprimal}
			\textbf u^{(k)}(x_j) = \textbf u^{(\ell)}(x_j)
			\quad
			\foralls
			j=1,\ldots,J
			\mbox{ and }
			\ell,k\in\mathcal N_x(j),
			\mbox{ and}
	\end{equation}
	\item the integrals of the normal components of the velocity, i.e.,
	\begin{equation}\nonumber
			\int_{\Gamma^{(k,\ell)}} \mathbf u^{(k)} \cdot \mathbf n^{(k)}
			\mathrm d s
			=
			-
			\int_{\Gamma^{(k,\ell)}} \mathbf u^{(\ell)} \cdot \mathbf n^{(\ell)}
			\mathrm d s
			\quad
			\foralls
			k=1,\ldots,K \mbox{ and }\ell\in\mathcal N_\Gamma(k).
	\end{equation}
\end{itemize}
Since the constraint~\eqref{eq:cornerprimal} is vector-valued, there are actually 2 primal degrees of freedom for each corner. Overall, for patches that do not contribute to the Dirichlet boundary, there are 12 primal degrees of freedom related to the continuity conditions. The matrix $\kbox{C_\C}$ evaluates the primal degrees of freedom associated to the patch $\Omega^{(k)}$;
thus the relation $\kbox{C_\C} \kbox{\underline{\mathbf{u}}_\Gamma} = 0$
guarantees that the primal degrees of freedom vanish.

Additionally, we introduce primal degrees of freedom for the pressure in order to be able to realize the condition that the average pressure vanishes (\emph{averaging conditions}). This is done by fixing the average pressure on each patch individually to zero and by allowing patchwise constant pressure functions in the primal problem. The matrix $\kbox{C_\A}$ evaluates the average pressure on the patch, i.e.,
\begin{equation}\label{eq:cadef}
			\kbox{C_\A} \kbox{\underline{p}}
			=
			|\Omega^{(k)}|^{-1} \int_{\Omega^{(k)}} p^{(k)}(x) \,\mathrm dx.
\end{equation}
Thus, the relation  $\kbox{C_\A} \kbox{\underline{p}} = 0$
guarantees that the average pressure vanishes.

The corresponding Lagrangian multipliers are denoted by $\kbox{\mu_\C}$ and $\kbox{\mu_\A}$. So, we obtain
\begin{equation*}
    \bar{A}^{(k)}\bar{\underline{\mathbf{x}}}^{(k)} :=
		\begin{pmatrix}
		  \kbox{K_{\Gamma\Gamma}}  &\kbox{K_{\Gamma\I}}  & \kbox{D_\Gamma}^\top &0              &\kbox{C_\C}^\top\\
		  \kbox{K_{\I\Gamma}}      &\kbox{K_{\I\I}}      & \kbox{D_\I}^\top    &0              &0\\
		  \kbox{D_\Gamma}          &\kbox{D_\I}          & 0             &\kbox{C_\A}^\top&0\\
		  0                        &0                    & \kbox{C_\A}    &0              &0\\
		  \kbox{C_\C}              &0                    &0              &0              &0\\
		\end{pmatrix}
        \begin{pmatrix}
          \kbox{\underline{\mathbf{u}}_\Gamma}\\
          \kbox{\underline{\mathbf{u}}_\I}\\
          \kbox{\underline{p}}\\
          \kbox{\mu_\A}\\
          \kbox{\underline{\mu}_\C}
        \end{pmatrix}
       \stackrel{!}{=}
        \begin{pmatrix}
          \kbox{\underline{\mathbf{f}}_\Gamma} \\
          \kbox{\underline{\mathbf{f}}_\I} \\
          0\\
          0\\
          0
        \end{pmatrix}=:\bar{\underline{\mathbf{b}}}^{(k)}.
\end{equation*}
The continuity of the velocity between the patches is enforced 
by the matrices $\kbox{B}$. The term
\[
		\sum_{k=1}^K \kbox{B} \underline{\mathbf u}_\Gamma^{(k)}
\]
evaluates to a vector containing the differences of the coefficients of any two matching
(cf. Assumption~\ref{assumption:conforming}) basis functions. Here, we do not include
the vertex values (see Figure~\ref{fig:nonredundant}) since these are primal
degrees of freedom anyway. Note that the constraint matrices $\kbox{B}$ are redundant
to the condition on the integrals of the normal components of the velocity over the edges.
The relation
\[
  \sum_{k=1}^K \kbox{B} \underline{\mathbf u}_\Gamma^{(k)} = 0
  \quad\text{or}\quad
  \sum_{k=1}^K \kbox{\bar{B}} \underline{\bar{\mathbf x}}^{(k)} =0
  \quad\text{using}\quad
		\kbox{\bar{B}}:=
		\begin{pmatrix}
 		 \kbox{B} & 0& 0 & 0 &0
		\end{pmatrix}
\]
guarantees the continuity of the velocity function across the patches.
\begin{figure}
	\begin{center}
	  \begin{tikzpicture}
      \tikzset{
        schraffiert/.style={pattern=north east lines,pattern color=#1},
        schraffiert/.default=black
      }

			\fill[schraffiert=gray!80] (-0.2,-0.1) -- (1.6,-0.1) -- (1.6,1.7) -- (-0.2,1.7);
			\fill[schraffiert=gray!80] (-0.2,-0.9) -- (1.6,-0.9) -- (1.6,-2.7) -- (-0.2,-2.7);
			\fill[schraffiert=gray!80] (4.2,-0.1) -- (2.4,-0.1) -- (2.4,1.7) -- (4.2,1.7);
			\fill[schraffiert=gray!80] (4.2,-0.9) -- (2.4,-0.9) -- (2.4,-2.7) -- (4.2,-2.7);

			\draw (-0.2,-0.1) -- (1.6,-0.1) -- (1.6,1.7) node at (0.7,0.8) {$\Omega^{(1)}$}; 
			\draw (-0.2,-0.9) -- (1.6,-0.9) -- (1.6,-2.7) node at (0.7,-1.85) {$\Omega^{(2)}$}; 
			\draw (4.2,-0.1) -- (2.4,-0.1) -- (2.4,1.7) node at (3.4,0.8) {$\Omega^{(3)}$}; 
			\draw (4.2,-0.9) -- (2.4,-0.9) -- (2.4,-2.7) node at (3.4,-1.85) {$\Omega^{(4)}$}; 
			
			\draw (1.6,-0.1) node[circle, fill, inner sep = 2pt] (A1) {};
			\draw (1.6,0.75) node[circle, fill, inner sep = 2pt] (A2) {};
			\draw (1.6,1.5) node[circle, fill, inner sep = 2pt] (A3) {};
			\draw (0.75,-0.1) node[circle, fill, inner sep = 2pt] (A4) {};
			\draw (0,-0.1) node[circle, fill, inner sep = 2pt] (A5) {};
			
			\draw (2.4,-0.1) node[circle, fill, inner sep = 2pt] (B1) {};
			\draw (2.4,0.75) node[circle, fill, inner sep = 2pt] (B2) {};
			\draw (2.4,1.5) node[circle, fill, inner sep = 2pt] (B3) {};
			\draw (3.25,-0.1) node[circle, fill, inner sep = 2pt] (B4) {};
			\draw (4,-0.1) node[circle, fill, inner sep = 2pt] (B5) {};
			
			\draw (1.6,-0.9) node[circle, fill, inner sep = 2pt] (C1) {};
			\draw (1.6,-1.75) node[circle, fill, inner sep = 2pt] (C2) {};
			\draw (1.6,-2.5) node[circle, fill, inner sep = 2pt] (C3) {};
			\draw (0.75,-0.9) node[circle, fill, inner sep = 2pt] (C4) {};
			\draw (0,-0.9) node[circle, fill, inner sep = 2pt] (C5) {};
			
			\draw (2.4,-0.9) node[circle, fill, inner sep = 2pt] (D1) {};
			\draw (2.4,-1.75) node[circle, fill, inner sep = 2pt] (D2) {};
			\draw (2.4,-2.5) node[circle, fill, inner sep = 2pt] (D3) {};
			\draw (3.25,-0.9) node[circle, fill, inner sep = 2pt] (D4) {};
			\draw (4,-0.9) node[circle, fill, inner sep = 2pt] (D5) {};
			
			\draw[<->, line width = 1pt, latex-latex]
			(A2) edge (B2) (A3) edge (B3)
			(A4) edge (C4) (A5) edge (C5)
			(B4) edge (D4) (B5) edge (D5)
			(C2) edge (D2) (C3) edge (D3);
			\end{tikzpicture}
	  \caption{Enforcing continuity of the velocity space. The corners are excluded.}
      \label{fig:nonredundant}
	\end{center}
\end{figure}
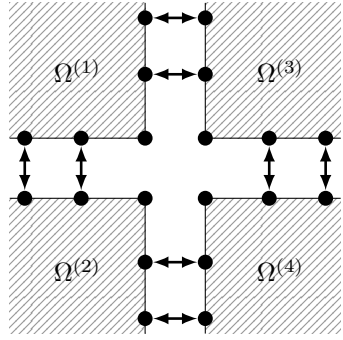

Moreover, we introduce the primal problem, i.e., the global problem for the primal degrees of freedom. We use a $A^{(k)}$-orthogonal basis for the primal degrees of freedom. This basis is represented in terms of the basis functions of the basis for $\mathbf V_\Gamma^{(k)}\times \mathbf V_\I^{(k)}\times Q^{(k)}$ using the matrix $\Psi^{(k)}$, which are the solution of the system
\begin{equation}\label{eq:basisdef}
		\begin{pmatrix}
		  \kbox{K_{\Gamma\Gamma}}  &\kbox{K_{\Gamma\I}}  & \kbox{D_\Gamma}^\top &0              &\kbox{C_\C}^\top\\
		  \kbox{K_{\I\Gamma}}      &\kbox{K_{\I\I}}      & \kbox{D_\I}^\top    &0              &0\\
		  \kbox{D_\Gamma}          &\kbox{D_\I}          & 0             &\kbox{C_\A}^\top&0\\
		  0                        &0                    & \kbox{C_\A}    &0              &0\\
		  \kbox{C_\C}              &0                    &0              &0              &0\\
		\end{pmatrix}
		\underbrace{
		\begin{pmatrix}
					\kbox{\Psi_{\Gamma\A}}   & \kbox{\Psi_{\Gamma\C}} \\
					\kbox{\Psi_{\I\A}}       & \kbox{\Psi_{\I\C}} \\
					\kbox{\Psi_{p\A}}        & \kbox{\Psi_{p\C}} \\
					\kbox{\Phi_{\C\A}}       & \kbox{\Phi_{\C\C}} \\
					\kbox{\Phi_{\A\A}}       & \kbox{\Phi_{\A\C}} \\
		\end{pmatrix}
		}_{\displaystyle 
		\hspace{-4em}
		\begin{pmatrix}
		 \kbox{\Psi}\\\kbox{\Phi}
		\end{pmatrix}:=
		\begin{pmatrix}
		 \kbox{\Psi_\A}&\kbox{\Psi_\C}\\\kbox{\Phi_\A}&\kbox{\Phi_\C}
		\end{pmatrix}:=
		\hspace{-4em}
		}
		=
		\begin{pmatrix}
					0 & 0 \\
					0 & 0 \\
					0 & 0 \\
					\kbox{R_\A} & 0 \\
					0 & \kbox{R_\C}\\
		\end{pmatrix},
\end{equation}
where $\kbox{R_\A}$ and $\kbox{R_\C}$ are boolean matrices that select the primal degrees of freedom which are active on the patch~$\Omega^{(k)}$.
We have $\kbox{R_\A}\in \mathbb R^{1\times K}$ and
$\kbox{R_\C}\in \mathbb{R}^{N^{(k)}_{\Pi}\times N_\Pi}$, where $N^{(k)}_{\Pi}$ is the number of primal degrees of freedom corresponding to the continuity condition (thus $N^{(k)}_{\Pi}=12$ if the patch does not contribute to the Dirichlet boundary) and $N_\Pi$ is the overall number of primal degrees of freedom associated to the continuity condition.

We define the system matrix, right-hand side and jump matrix for the primal problem as
\[
A_\Pi := \sum^K_{k=1} \kbox{\Psi}^\top \kbox{A}\kbox{\Psi},\quad
\mathbf{b}_\Pi := \sum^K_{k=1} \kbox{\Psi}^\top \kbox{\underline{\mathbf{b}}}
\quad \text{and}\quad
B_\Pi := \sum^K_{k=1} \kbox{B}\kbox{\Psi}.
\]
So far, we have a pressure averaging condition for the patch-local problems and the primal problem allows for patchwise constant pressure modes. So, in order to obtain unique solvability of the problem, we need to add a global condition that guarantees that the average pressure vanishes. So, we augment the primal system and obtain the following primal system
\[
    \bar{A}_\Pi \underline{\bar{\mathbf{x}}}_{\Pi} :=
	     \begin{pmatrix}
          A_\Pi &C_\Pi^\top\\
          C_\Pi & 0
		 \end{pmatrix}
		 \begin{pmatrix}
          \underline{\mathbf{x}}_{\Pi} \\
          \underline{\mu}_{\Pi} 
        \end{pmatrix}
        \stackrel{!}{=}
        \begin{pmatrix}
          \underline{\mathbf{b}}_\Pi \\
          0
        \end{pmatrix}
        =:
        \bar{\underline{\mathbf{b}}}_\Pi,
\]
where $C_\Pi\in\mathbb R^{1\times (N_\Pi+K)}$
is such that the relation $C_\Pi\, \underline{\mathbf x}_\Pi=0$
guarantees that the average of the pressure vanishes. Correspondingly,
we define $\bar{B}_\Pi:=\begin{pmatrix} B_\Pi & 0 \end{pmatrix}$.
So, we are finally able to write down the overall IETI-DP system:
\[
		\begin{pmatrix}
				\kkbox{\bar{A}}{1}   &&&& \kkbox{\bar{B}}{1}^\top \\
				& \ddots              &&& \vdots                  \\
				&& \kkbox{\bar{A}}{K}  && \kkbox{\bar{B}}{K}^\top \\
				&&& \bar{A}_\Pi         & \bar{B}_\Pi^\top        \\
		    \kkbox{\bar{B}}{1} & \hdots & \kkbox{\bar{B}}{K} & \bar{B}_\Pi  \\
		\end{pmatrix}
		\begin{pmatrix}
			\kkbox{\underline{\bar{\mathbf x}}}1\\
			\vdots\\
			\kkbox{\underline{\bar{\mathbf x}}}K\\
			\underline{\bar{\mathbf x}}_\Pi \\
			\underline{\lambda}
		\end{pmatrix}
		=
		\begin{pmatrix}
			\kkbox{\underline{\bar{\mathbf b}}}1\\
			\vdots\\
			\kkbox{\underline{\bar{\mathbf b}}}K\\
			\underline{\bar{\mathbf b}}_\Pi \\
			0		
		\end{pmatrix}		
		.
\]
For solving this linear system, we take the Schur complement with
respect to the Lagrange multipliers $\underline{\lambda}$. 
This means that we solve
\begin{equation}
  \label{eq:lambdaSys}
  \bar{F}\underline{\lambda} = \underline{g},
\end{equation}
where
\begin{equation*}
	\bar{F} := \bar{F}_\Pi+\sum^{K}_{k = 1} \bar{F}^{(k)},
	\quad 
	\bar{F}_\Pi :=  \bar{B}_\Pi \bar{A}^{-1}_\Pi \bar{B}_\Pi^\top,
	\quad
	\kbox{\bar F} := \kbox{\bar B} \kbox{\bar{A}}^{-1}  \kbox{\bar B}^{\top}
\end{equation*}
and
\[
\underline{g}:= \bar{B}_\Pi \bar{A}^{-1}_\Pi\bar{\underline{\mathbf{b}}}_\Pi + \sum^{K}_{k = 1}  \bar{B}^{(k)} \kbox{\bar{A}}^{-1}\bar{\underline{\mathbf{b}}}^{(k)}
\]
The patch-local solutions are then recovered by
\[
\underline{\bar{\mathbf{x}}}_\Pi = \bar{A}_\Pi^{-1}\left(\underline{\bar{\mathbf{b}}}_\Pi-\bar B_\Pi^{\top}\underline{\lambda}\right)
\quad\mbox{and}\quad
\kbox{\underline{\bar{\mathbf{x}}}} = \kbox{\bar{A}}^{-1}\left(\underline{\bar{\mathbf{b}}}^{(k)}- \kbox{\bar B}^{\top}\underline{\lambda}\right).
\]
The final solution is then obtained by distributing $\underline{\bar{\mathbf{x}}}_\Pi$ to the patches using the matrices $\Psi^{(k)}$.

We solve the system~\eqref{eq:lambdaSys} with a conjugate gradient solver. (We will show in Section~\ref{sec:5} that $\bar F$ is indeed positive semidefinite.) The conjugate gradient solver is preconditioned with the scaled Dirichlet preconditioner. Usually, the scaled Dirichlet preconditioner refers to the local solution of Dirichlet problems of the corresponding differential equation, which would mean that we should consider the Stokes equations. However, this is
not necessary. Indeed, the local Dirichlet problems are solved to realize the $H^{1/2}$-norm. For this purpose, it is sufficient to only consider local Dirichlet problems of the Poisson equation. A recent numerical study, cf. Ref.~\refcite{sogn2021dual}, has shown that this approach is not only simpler, but also leads to better convergence behavior.
Thus, we define the local Schur complements via
\begin{equation}\label{eq:skdef}
		\kbox{S_K} :=
		\kbox{K_{\Gamma\Gamma}}-
		\kbox{K_{\Gamma\I}}  
		\kbox{K_{\I\I}}^{-1}
		\kbox{K_{\I\Gamma}}.
\end{equation}
Then, the scaled Dirichlet preconditioner is given by
\begin{equation}
  \label{eq:precM}
		M_{\mathrm{sD}} := \sum_{k=1}^K \kbox{B} \scaling_k^{-1} \kbox{S_{K}} \scaling_k^{-1} \kbox{B}^\top,
\end{equation}
where $\scaling_k:=2I$ is set up based on the principle of multiplicity scaling.

\section{Condition number analysis for the IETI solver}
\label{sec:5}

The convergence rates of the conjugate gradient solver are estimated
based on the condition number of the preconditioned system $ M_{\mathrm{sD}} \bar F$. Following the framework introduced in Ref.~\refcite{MandelDohrmannTezaur:2005a}, we rewrite the whole problem equivalently as a formulation only living to the skeleton.
First, we define the skeleton formulation associated to the main saddle point matrix $\kbox{A}$ via
\begin{equation}\label{eq:sgammadef}
	\kbox{S_A} :=
	\kbox{K_{\Gamma\Gamma}} -
	\begin{pmatrix}
			\kbox{K_{\Gamma\I}}         & \kbox{D_\Gamma}^\top        & 0
	\end{pmatrix}
	\begin{pmatrix}
		  \kbox{K_{\I\I}} &\kbox{D_{\I}}^\top & 0 \\
		  \kbox{D_{\I}}   &0                  & \kbox{C_\A}^\top \\
		  0               &\kbox{C_\A}         & 0     
	\end{pmatrix}^{-1}
	\begin{pmatrix}
				\kbox{K_{\I\Gamma}} \\
				\kbox{D_{\Gamma}}   \\
				0
	\end{pmatrix}.
\end{equation}
Note that the inverse is well-defined due the constraint on the average pressure.

Next, we define corresponding function spaces.
Let $\mathbf W:=\mathbf W^{(1)} \times \cdots \times \mathbf W^{(K)}$
with
\[
		\mathbf W^{(k)} := \{ \mathbf v|_{\partial\Omega^{(k)}} \,:\, \mathbf v\in \mathbf V^{(k)} \}
\]
be the skeleton space and
$\mathcal H_k: \mathbf W^{(k)} \rightarrow \mathbf V^{(k)}$
be the discrete harmonic extension, i.e.,
$\mathcal H_k \mathbf w^{(k)}\in \mathbf V^{(k)}$ such that it minimizes
$|\mathcal H_k \mathbf w^{(k)}|_{H^1(\Omega^{(k)})}$
under the constraint $\left(\mathcal H_k \mathbf w^{(k)}\right)|_{\partial\Omega^{(k)}}=\mathbf w^{(k)}$.
The function space $\widetilde{\mathbf W}_\Delta:=
\widetilde{\mathbf W}_\Delta^{(1)}\times \cdots \times \widetilde{\mathbf W}_\Delta^{(K)}$,
where
\[
		\widetilde{\mathbf W}_\Delta^{(k)} :=
		\left\{
				\mathbf w^{(k)} \in \mathbf W^{(k)}
				:
				\begin{array}{lr}
				\mathbf w^{(k)}(x_j) = 0
				&\hspace{-4em}\foralls j\in\{1,\ldots,J\} \mbox{ with } k\in\mathcal N_x(j)
				\\
				\int_{\Gamma^{(k,\ell)}} \mathbf w^{(k)} \cdot \mathbf n^{(k)} \; \mathrm d s = 0
				&\foralls \ell\in\mathcal N_\Gamma(k)
				\end{array}
		\right\},
\]
represents the functions that satisfy the primal degrees of freedom homogeneously. The space $\widetilde{\mathbf W}$, in which all the approximate solutions live in, is given by
\[
		\widetilde{\mathbf W} :=
		\left\{
				\mathbf w \in \mathbf W
				:
				\begin{array}{lr}
				\mathbf w^{(k)}(x_j)=\mathbf w^{(\ell)}(x_j)
				&\hspace{-5em}\foralls k,\ell\in\mathcal N_x(j)
				 \mbox{ with } j\in \{1,\ldots,J\},
				\\
				\int_{\Gamma^{(k,\ell)}} (\mathbf w^{(k)}-\mathbf w^{(\ell)}) \cdot \mathbf n^{(k)}\, \mathrm d s = 0
				&\foralls k \mbox{ and } \ell\in\mathcal N_\Gamma(k),
				\\
				\int_{\partial\Omega^{(k)}} \mathbf w^{(k)} \cdot \mathbf n^{(k)}\, \mathrm d s = 0
				&\quad \foralls k\in\{1,\ldots,K\}
				\end{array}
		\right\},
\]
where $\mathbf w=(\mathbf w^{(1)},\cdots,\mathbf w^{(K)})$.
The first two lines in definition of this space refer to continuity constraints on the
velocity, both on the corner values of the velocities and the integrals of the normal components
of the velocity on all of the edges. Since the original formulation of the discretized Stokes problem~\eqref{eq:discstokes} uses a continuous discretization space for the velocity, these continuity
conditions are obviously satisfied by the solution. However, the definition also introduces a third class of constraints of the form
\begin{equation}\label{eq:w:cond}
		\int_{\partial \Omega^{(k)}} \mathbf w^{(k)}\cdot \mathbf n^{(k)}\, \mathrm d s = 0
		\quad
		\foralls
		k=1,\ldots,K,
\end{equation}
i.e., that the inflow equals the outflow on each patch.
Let $\mathbf u=(\mathbf u^{(1)},\ldots\mathbf u^{(K)})$ be the exact solution with skeleton representation $\mathbf w^{(k)}:=\mathbf u^{(k)}|_{\partial\Omega{(k)}}$ and let 
\[
	\chi_{\Omega^{(k)}} (x) := \begin{cases} 1 & \mbox{ if } x\in \Omega^{(k)}, \\
	                                         0 & \mbox{ otherwise.} \end{cases} 
\]
be the characteristic functions. By integration by parts and since the functions
$\tilde{\chi}_{\Omega^{(k)}}(x):= \chi_{\Omega^{(k)}}(x)-|\Omega^{(k)}|/|\Omega|$ are contained in $Q$, we know from~\eqref{eq:discstokes} that
\[
		\int_{\partial \Omega^{(k)}} \mathbf w^{(k)}\cdot \mathbf n^{(k)}\, \mathrm d s
		=
		\int_{\Omega^{(k)}} \nabla \cdot \mathbf u^{(k)} \, \mathrm d x
		=
		\frac{|\Omega^{(k)}|}{|\Omega|}
		\int_{\Omega} \nabla \cdot \mathbf u \, \mathrm d x
		=
		\frac{|\Omega^{(k)}|}{|\Omega|}
		\int_{\partial \Omega} \mathbf u\cdot \mathbf n\, \mathrm d s,
\]
which vanishes due to the Dirichlet boundary conditions. This shows that~\eqref{eq:w:cond} holds for the solution. The space $\widetilde{\mathbf W}_\Pi$ is the orthogonal complement to $\widetilde{\mathbf W}_\Delta$
in $\widetilde{\mathbf W}$, i.e., we define
\[
		\widetilde{\mathbf W}_\Pi :=
		\left\{
				\mathbf w \in \widetilde{\mathbf W}
				:
				\sum_{k=1}^K 
				(\kbox{\underline{\mathbf w}},\kbox{\underline{\mathbf v}})_{S_A^{(k)}}=0
				\mbox{ for all } \mathbf v\in \widetilde{\mathbf W}_\Delta
		\right\}.
\]
Here and in what follows, for any function, say $\kbox{\mathbf w}$
in $\kbox{\mathbf W}$, 
$\kbox{\widetilde{\mathbf W}}$, $\kbox{\widetilde{\mathbf W}_\Delta}$
or $\kbox{\widetilde{\mathbf W}_\Pi}$, the corresponding underlined
symbol, here $\kbox{\underline{\mathbf w}}$, denotes the representation
of the corresponding function with respect to the basis for the space
$\kbox{\mathbf W}$. So, functions in the all of these spaces are represented with respect to the same basis.

For the analysis, we introduce the following lemma that allows to write expressions
involving inverses of matrices as suprema.
\begin{lemma}\label{lem:supmat}
	Let $A\in \mathbb R^{n\times n}$ be a symmetric positive definite matrix
	and let $B\in \mathbb R^{m_1\times n}$, $C\in \mathbb R^{m_2\times n}$ and
	$D\in \mathbb R^{m_3\times m_2}$.
	Then, we have 
	\begin{equation}\label{eq:supmat0}
		\|\underline \lambda \|_{M_0} = \sup_{\underline w \in W_0}
				\frac{(B \underline w, \underline \lambda)_{\ell^2}}{\| \underline w\|_A },
		\;\mbox{where}\;
		M_0 := B A^{-1} B^\top
	\end{equation}
	and $W_0 := \mathbb R^{n}$,
	\begin{equation}\label{eq:supmat1}
		\|\underline \lambda \|_{M_1} = \sup_{\underline w \in W_1}
				\frac{(B \underline w, \underline \lambda)_{\ell^2}}{\| \underline w\|_A },
		\;\mbox{where}\;
		M_1 :=
			\begin{pmatrix} B & 0 \end{pmatrix}
			\begin{pmatrix} A & C^\top \\ C & 0 \end{pmatrix}^{-1}
			\begin{pmatrix} B^\top \\ 0 \end{pmatrix}
	\end{equation}
	and $W_1 := \{ \underline w \in \mathbb R^{n} \, :\, C \underline w = 0 \}$, 
	and
	\begin{equation}\label{eq:supmat2}
		\|\underline \lambda \|_{M_2} = \sup_{\underline w \in W_2}
				\frac{(B \underline w, \underline \lambda)_{\ell^2}}{\| \underline w\|_A },
		\;\mbox{where}\;
		M_2 :=
			\begin{pmatrix} B & 0 & 0\end{pmatrix}
			\begin{pmatrix} A & C^\top & 0 \\ C & 0 & D^\top \\ 0 & D & 0 \end{pmatrix}^{-1}
			\begin{pmatrix} B^\top \\ 0 \\ 0 \end{pmatrix}
	\end{equation}
	and $W_2 := \{ \underline w \in \mathbb R^{n} \, :\, D \underline \mu = 0 \Rightarrow
	\underline \mu^\top C \underline w = 0 \mbox{ for all }\underline \mu\in \mathbb R^{m_2} \}$.
\end{lemma}
The statements~\eqref{eq:supmat0} and \eqref{eq:supmat1} are 
standard, cf. Ref.~\refcite{fortin1991mixed}, Chapter~II, §~1.1. For completeness,
we give a proof of~\eqref{eq:supmat2} in the Appendix.

First we note the equivalence of $\kbox{S_A}$ and $\kbox{S_K}$.
\begin{lemma}\label{lem:h1like}
	We have
	\[
		\kbox{\underline{\mathbf w}}^\top S_K^{(k)} \kbox{\underline{\mathbf w}}
		\le
		\kbox{\underline{\mathbf w}}^\top S_A^{(k)} \kbox{\underline{\mathbf w}}
		\le 3\frac{\delta^2}{\beta_k^{2}} 
		\kbox{\underline{\mathbf w}}^\top S_K^{(k)} \kbox{\underline{\mathbf w}}
	\]
	for all skeleton functions $\kbox{\mathbf w}\in \kbox{\mathbf W}$,
	where $S_K^{(k)}$ is as defined in~\eqref{eq:skdef}.
\end{lemma}
\begin{proof}
Recalling~\eqref{eq:cadef}, we have that
\[
		C_\A^{(k)} = |\Omega^{(k)}|^{-1}\; \underline 1^{\top} \kbox{M_p},
\]
where $\kbox{M_p}$ is the mass matrix that is obtained from discretizing $(\cdot,\cdot)_{L^2(\Omega^{(k)})}$ with the basis functions in the basis for $Q^{(k)}$ and
$\underline 1=(1,\cdots,1)^\top$ is a vector of ones of the corresponding size. Certainly, we have $\kbox{C_\A}\underline 1=1$.
We observe that the Gauss rule yields
\[
	 \kbox{\underline{\mathbf v}_\I}^\top \kbox{D_\I}^\top \underline 1
		=( \nabla \cdot \kbox{\mathbf v_\I},1)_{L^2(\Omega^{(k)})}
		=( \kbox{\mathbf v_\I}\cdot\kbox{\mathbf n},1)_{L^2(\partial \Omega^{(k)})}
		=0
\]
for all $\kbox{\mathbf v_\I}\in \kbox{\mathbf V_\I}$,
i.e., $\kbox{D_\I}^\top \underline 1 =0$.
Define $\kbox{W}:=\kbox{D_\I} \kbox{K_{\I\I}}^{-1} \kbox{D_\I}^\top+\kbox{C_\A}^\top \kbox{C_\A}$
and observe that $\kbox{W}^{-1}\kbox{C_\A}^{\top}=\underline 1$. Using these identities,
it is easily verified that
\begin{align*}
		&
			\begin{pmatrix}
		  \kbox{K_{\I\I}} &\kbox{D_{\I}}^\top & 0 \\
		  \kbox{D_{\I}}   &0                  & \kbox{C_\A}^\top \\
		  0               &\kbox{C_\A}         & 0     
	\end{pmatrix}^{-1}
		\\
		&=\begin{pmatrix}
				  \kbox{K_{\I\I}}^{-1} - \kbox{K_{\I\I}}^{-1} \kbox{D_\I}^\top \kbox{W}^{-1} \kbox{D_\I} \kbox{K_{\I\I}}^{-1}
				& \kbox{K_{\I\I}}^{-1} \kbox{D_\I}^\top \kbox{W}^{-1} 
				& 0 \\
				  \kbox{W}^{-1} \kbox{D_\I} \kbox{K_{\I\I}}^{-1} 
				& \underline 1\,\underline 1^\top-\kbox{W}^{-1}
				& \underline 1 \\
				  0
				& \underline 1^\top
				& 0
		\end{pmatrix}.
\end{align*}
Using the definition of $\kbox{S_A}$ and by expanding the products, we obtain
\begin{align*}
		\kbox{S_A} =\kbox{S_K}+\kbox{S_D}-\kbox{S_E},
\end{align*}
where 
$
		\kbox{S_K} =
		\kbox{K_{\Gamma\Gamma}}-
		\kbox{K_{\Gamma\I}}  
		\kbox{K_{\I\I}}^{-1}
		\kbox{K_{\I\Gamma}}
$
is as defined in~\eqref{eq:skdef} and
\begin{align*}
		\kbox{S_D} &:=
		(\kbox{K_{\Gamma\I}} \kbox{K_{\I\I}}^{-1} \kbox{D_\I}^\top - \kbox{D_\Gamma}^\top  )
		\kbox{W}^{-1}
		(\kbox{D_\I} \kbox{K_{\I\I}}^{-1} \kbox{K_{\I\Gamma}}  - \kbox{D_\Gamma}  ),\\
		\kbox{S_E} &:=
		\kbox{D_{\Gamma}}^\top \underline 1\, \underline 1^\top \kbox{D_\Gamma}
		.
\end{align*}
We first observe that
\begin{equation}\label{eq:skdesc}
\begin{aligned}
	\kbox{\underline{\mathbf w}}^\top
	\kbox{S_K}
	\kbox{\underline{\mathbf w}}
	&=
	\kbox{\underline{\mathbf w}}
	\begin{pmatrix}
		I & -   \kbox{K_{\Gamma\I}} \kbox{K_{\I\I}}^{-1}
	\end{pmatrix}
	\begin{pmatrix}
		\kbox{K_{\Gamma\Gamma}} & \kbox{K_{\Gamma\I}}  \\
		\kbox{K_{\I\Gamma}}     & \kbox{K_{\I\I}}
	\end{pmatrix}
	\begin{pmatrix}
		I \\ -  \kbox{K_{\I\I}}^{-1} \kbox{K_{\I\Gamma}} 
	\end{pmatrix}
	\kbox{\underline{\mathbf w}}
	\\&=
	| \mathcal H_k \kbox{\mathbf w} |_{H^1(\Omega^{(k)})}^2.
\end{aligned}
\end{equation}
Next, we observe that
\begin{align*}
		\kbox{S_D}
		& =
		\begin{pmatrix}
				I & - \kbox{K_{\Gamma\I}} \kbox{K_{\I\I}}^{-1} 
		\end{pmatrix}
		\begin{pmatrix}
				\kbox{D_\Gamma}^\top \\ \kbox{D_\I}^\top        
		\end{pmatrix}
		\kbox{W}^{-1}
		\begin{pmatrix}
				\kbox{D_\Gamma}  & \kbox{D_\I}       
		\end{pmatrix}
		\begin{pmatrix}
				\kbox{I} \\ -  \kbox{K_{\I\I}}^{-1} \kbox{K_{\I\Gamma}}
		\end{pmatrix}\\
		&= \mathcal H_k^\top \kbox{D}^\top \kbox{W}^{-1} \kbox{D} \mathcal H_k,
\end{align*}
where $	\kbox{D} := \begin{pmatrix}\kbox{D_\Gamma}        &  \kbox{D_\I} \end{pmatrix}$
is the overall divergence and $\mathcal H_k$ is the matrix representation of the discrete harmonic extension.
This means that we have using Lemma~\ref{lem:supmat}
\[
	\kbox{\underline{\mathbf w}}^\top \kbox{S_D} \kbox{\underline{\mathbf w}}^\top
	=
	\sup_{\kbox{q}\in Q^{(k)}}
	\frac{
		( \nabla\cdot (\mathcal H_k \kbox{\mathbf w}),  \kbox{q} )_{L^2(\Omega^{(k)})}^2,
		}{
			\| \kbox{\underline q}\|_{\kbox{W}}^2
		}.
\]
Using
\[
		\underline 1^\top \kbox{W} \underline 1
		=
		\sup_{\kbox{\mathbf v_\I}\in \kbox{\mathbf V_\I}}
		\underbrace{
		\frac{(\nabla \cdot \kbox{\mathbf v_\I},1)_{L^2(\Omega^{(k)})}^2}
		     {| \kbox{\mathbf v_\I} |_{H^1(\Omega^{(k)})}^2}
		}_{\displaystyle = 0}
		+
		|\Omega^{(k)}|^{-2}\left(\int_{\Omega^{(k)}} 1 \,\mathrm dx\right)^2
		=1,
\]
Using the Gauss rule and the choice $\kbox{q}=1$, we obtain
\[
\begin{aligned}
	&\kbox{\underline{\mathbf w}}^\top \kbox{S_D} \kbox{\underline{\mathbf w}}
	 \ge
	 ( \nabla\cdot (\mathcal H_k \kbox{\mathbf w}), 1 )_{L^2(\Omega^{(k)})}^2
	=
	(  (\mathcal H_k\kbox{\mathbf w}) \cdot \kbox{\textbf n},1)_{L^2(\partial\Omega^{(k)})}^2
	\\
	&\quad
	=
	(  \kbox{\mathbf w} \cdot \kbox{\textbf n},1)_{L^2(\partial\Omega^{(k)})}^2
	=
	(\nabla \cdot \kbox{\mathbf v_\Gamma},1)_{L^2(\Omega^{(k)})}^2
	= 
	\big( \underline 1^\top D_\Gamma \kbox{\underline{\mathbf v}_\Gamma}     \big)^2
	\\&\quad=
	\kbox{\underline{\mathbf v}_\Gamma}^\top \kbox{S_E} \kbox{\underline{\mathbf v}_\Gamma}=
	\kbox{\underline{\mathbf w}}^\top \kbox{S_E} \kbox{\underline{\mathbf w}},
\end{aligned}
\]
where $\kbox{\mathbf v_\Gamma}\in \kbox{\mathbf V_\Gamma}$ such that
$\kbox{\mathbf v_\Gamma}|_{\partial\Omega^{(k)}}=\kbox{\mathbf w}$
(and thus $\kbox{\underline{\mathbf v}_\Gamma}=\kbox{\underline{\mathbf w}}$).
This shows
\[
	\kbox{\underline{\mathbf w}}^\top \kbox{S_A} \kbox{\underline{\mathbf w}}
	=
	  \kbox{\underline{\mathbf w}}^\top
	  (\kbox{S_K} + \kbox{S_D}- \kbox{S_E})
	  \kbox{\underline{\mathbf w}}
	\ge
	  \kbox{\underline{\mathbf w}}^\top \kbox{S_K} \kbox{\underline{\mathbf w}}
\]
and thus the desired bound from below.

The definition of $\kbox{W}$, local inf-sup stability~\eqref{eq:infsuplocal} and $\beta_k \le \delta = \sqrt2$ give
for all $\kbox{p} \in Q^{(k)}$:
\begin{align*}
	\kbox{\underline p}^\top \kbox{W} \kbox{\underline p}
	&= 
	\sup_{\mathbf{v}_\I \in \mathbf{V}_\I^{(k)}}
			\frac{ (\nabla \cdot \mathbf{v}_\I, \kbox{p})_{L^2(\Omega^{(k)})}^2 }{ |\mathbf{v}_\I|_{H^1(\Omega^{(k)})}^2  }
	+
	|\Omega^{(k)}|^{-2}
	\left(
	\int_{\kbox{\Omega}} \kbox{p} \; \mathrm d x
	\right)^2
	\\&\ge
	\beta_{k}^2 \inf_{q\in\mathbb R} \|\kbox{p}-q\|_{L^2(\Omega^{(k)})}^2
	+
	|\Omega^{(k)}|^{-2}
	\left(
	\int_{\kbox{\Omega}} \kbox{p} \;  \mathrm d x
	\right)^2
	\\&\ge \min\{\beta_k^2,1\} \|\kbox{p}\|_{L^2(\Omega^{(k)})}^2
	\ge \frac{\beta_k^2}{2} \|\kbox{p}\|_{L^2(\Omega^{(k)})}^2
	.
\end{align*}
Using~\eqref{eq:boundedness} and~\eqref{eq:skdesc}, we further obtain
\begin{align*}
	\kbox{\underline{\mathbf w}}^\top \kbox{S_D} \kbox{\underline{\mathbf w}}
	&\le 2\beta_k^{-2}
	\sup_{\kbox{p}\in Q^{(k)}}
	\frac{
		( \nabla\cdot (\mathcal H_k \kbox{\mathbf w}),  \kbox{p} )_{L^2(\kbox{\Omega})}^2
		}{
			\|\kbox{p}\|_{L^2(\kbox{\Omega})}^2
		}
	\\&\le
	2\beta_k^{-2}\delta^2 | \mathcal H_k \kbox{\mathbf w} |_{H^1(\Omega^{(k)})}^2
	=
	2\beta_k^{-2}\delta^2
	\kbox{\underline{\mathbf w}}^\top \kbox{S_K} \kbox{\underline{\mathbf w}}
	.
\end{align*}
Using $\kbox{\underline{\mathbf w}}^\top \kbox{S_E} \kbox{\underline{\mathbf w}}\ge0$ and $1\le\beta_k^{-2}\delta^2$, we finally have 
\begin{align*}
	\kbox{\underline{\mathbf w}}^\top \kbox{S_A} \kbox{\underline{\mathbf w}}
	&=
	 \kbox{\underline{\mathbf w}}^\top (\kbox{S_K}+\kbox{S_D}-\kbox{S_E}) \kbox{\underline{\mathbf w}}
	\le
	3\beta_k^{-2}\delta^2
	\kbox{\underline{\mathbf w}}^\top \kbox{S_K} \kbox{\underline{\mathbf w}},
\end{align*}
which finishes the proof.
\end{proof}

Recalling
\[
		\bar F = \bar F_\Pi+\sum_{k=1}^K \bar F^{(k)},
\]
we estimate $\bar F^{(k)}$ and $\bar F_\Pi$ separately using Lemmas~\ref{lem:bar:f} and~\ref{lem:bar:fpi} below.
\begin{lemma}
		\label{lem:bar:f}
		For $k=1,\ldots,K$,
		the matrices $\kbox{\bar F}$ are symmetric and positive semidefinite and 
		satisfy
		\[
				\| \underline \lambda \|_{\kbox{\bar F}}
				=
				\sup_{\mathbf w^{(k)} \in \widetilde{\mathbf W}_\Delta^{(k)}}
				\frac{(\kbox{B}^\top \underline \lambda,\underline{\mathbf w}^{(k)})_{\ell^2}}
				{\|\underline{\mathbf w}^{(k)}\|_{S_A^{(k)}}}.
		\]
\end{lemma}
\begin{proof}
	Using the definition of $\bar F^{(k)}$, the block structure of $\bar A^{(k)}$,
	the fact that we can reorder the entries in $\bar A^{(k)}$, block Gaussian elimination and~\eqref{eq:sgammadef}, we obtain
	\begin{align*}
			\kbox{\bar F}  &= 
				\kbox{\bar B} \kbox{\bar A}^{-1} \kbox{\bar B}^\top 
				\\
				&
				=
				\begin{pmatrix}
					\kbox{B}^\top \\ 0 \\ 0 \\0 \\ 0
				\end{pmatrix}^\top
				\begin{pmatrix}
					\kbox{K_{\Gamma\Gamma}} &\kbox{C_\C}^\top    &\kbox{K_{\Gamma\I}}  & \kbox{D_{\Gamma}}^\top &0\\
				  \kbox{C_\C}             &0                   &0                    & 0                      &0\\
				  \kbox{K_{\I\Gamma}}     &0                   &\kbox{K_{\I\I}}      & \kbox{D_{\I}}^\top     &0\\
				  \kbox{D_\Gamma}         &0                   &\kbox{D_\I}          & 0                      &\kbox{C_\A}^\top\\
				  0                       &0                   &0                    & \kbox{C_\A}             &0\\
				\end{pmatrix}^{-1}
				\begin{pmatrix}
					\kbox{B}^\top \\ 0 \\ 0 \\0 \\ 0
				\end{pmatrix}
				\\
				&=
				\begin{pmatrix}
					\kbox{B} & 0 
				\end{pmatrix}
				\begin{pmatrix}
					\kbox{S_A}      &\kbox{C_\C}^\top\\
				  \kbox{C_\C}     &0            
				\end{pmatrix}^{-1}
				\begin{pmatrix}
					\kbox{B}^\top \\ 0 
				\end{pmatrix}
				.
		\end{align*}
		Since
		$\kbox{\mathbf w} \in \widetilde{\mathbf W}_{\Delta}^{(k)}$ if and only if
		$\kbox{C_\C} \kbox{\underline {\mathbf w}}=0$, Lemma~\ref{lem:supmat} gives the desired
		representation. This representation
		shows that $\kbox{\bar F}$ must be (symmetric and) positive definite.
\end{proof}

Before we discuss $\bar F_\Pi$, we observe that the primal basis functions associated to
the averaging conditions are just the constant pressure functions, precisely the pressure component of the $k$-th basis function $\chi_{\Omega^{(k)}}$, i.e., $1$ on the patch $\Omega^{(k)}$ and $0$ on all other patches. The velocity component vanishes on all patches.

\begin{lemma}\label{lemma:abasis}
	We have $\kbox{\Psi_{\Gamma\A}}=0$, $\kbox{\Psi_{\I\A}}=0$ and $\kbox{\Psi_{p\A}}=\underline 1 \kbox{R_\A}$.
\end{lemma}
\begin{proof}
	Since $\kbox{R_\A}=(0,\cdots,0,1,0,\cdots,0)$, where the non-zero coefficient is in
	the $k$-th column, we immediately obtain from the definition~\eqref{eq:basisdef} that
	all other columns of $\Psi_{\Gamma\A}^{(k)}$, $\Psi_{\I\A}^{(k)}$ and $\Psi_{p\A}^{(k)}$ vanish. Let
	$\underline\psi_{\Gamma\A}^{(k)}$, $\underline\psi_{\I\A}^{(k)}$ and $\underline\psi_{p\A}^{(k)}$ be the $k$-th
	column of these matrices. Then, we have
	\begin{equation}\label{eq:def:pressureprimals}
			\begin{pmatrix}
		  \kbox{K_{\Gamma\Gamma}}  &\kbox{K_{\Gamma\I}}  & \kbox{D_\Gamma}^\top &0              &\kbox{C_\C}^\top\\
		  \kbox{K_{\I\Gamma}}      &\kbox{K_{\I\I}}      & \kbox{D_\I}^\top    &0              &0\\
		  \kbox{D_\Gamma}          &\kbox{D_\I}          & 0             &\kbox{C_\A}^\top&0\\
		  0                        &0                    & \kbox{C_\A}    &0              &0\\
		  \kbox{C_\C}              &0                    &0              &0              &0\\
		\end{pmatrix}
			\begin{pmatrix}
				\kbox{\underline\psi_{\Gamma\A}}\\
				\kbox{\underline\psi_{\I\A}}\\
				\kbox{\underline\psi_{p\A}}\\
				\kbox{\rho}\\
				\kbox{\underline\mu}
			\end{pmatrix}
			=
			\begin{pmatrix}
				0\\
				0\\
				0\\
				1\\
				0
			\end{pmatrix}.
	\end{equation}
	Since the matrix $\kbox{\bar A}$ is non-singular, the system has a unique solution. 
	Choose $\kbox{\underline\psi_{\Gamma\A}}:=0$, $\kbox{\underline\psi_{\I\A}}:=0$, $\kbox{\underline\psi_{p\A}}:=\underline 1$,
	$\kbox{\rho}:=0$
	and $\kbox{\underline\mu}$ such that
	\[
			\kbox{\underline\mu}^\top \kbox{C_\C}\kbox{\underline{\mathbf v}_\Gamma} 
			=
			-
			\int_{\partial\Omega^{(k)}} \kbox{\mathbf v_\Gamma} \cdot \mathbf n^{(k)} \,\mathrm ds
			\qquad \foralls \kbox{\mathbf v_\Gamma}\in\kbox{\mathbf V_\Gamma},
	\]
	which is possible since $\kbox{C_\C}$ evaluates the primal degrees of freedom and
	the integrals of the normal components of the velocity variable on each edge are
	primal degrees of freedom. So, $\kbox{\underline\mu}$ is just 
	such that sum over the edges. Using the Gauss rule, we further have
	\begin{align*}
			\kbox{\underline\mu}^\top \kbox{C_\C}\kbox{\underline{\mathbf v}_\Gamma} 
			=
			-
			\int_{\Omega^{(k)}} \nabla \cdot \kbox{\mathbf v_\Gamma}  \,\mathrm dx
			=
			-
			\underline 1^\top \kbox{D_\Gamma} \kbox{\underline{\mathbf v}_\Gamma}
			\qquad \foralls \kbox{\mathbf v_\Gamma}\in \kbox{\mathbf V_\Gamma},
	\end{align*}
	which shows $\kbox{D_\Gamma}^\top \underline 1 + \kbox{C_\Gamma}^\top \kbox{\mu} = 0$. Analogously, we obtain $\kbox{D_\I}^\top \underline 1=0$.
	Since $\underline 1$ represents the constant function with value $1$
	and $\kbox{C_\A}$ evaluates the average, we have $\kbox{C_\A}\underline 1=1$.
	Using these results, it is easily verified that $(\kbox{\underline\psi_{\Gamma\A}},\kbox{\underline\psi_{\I\A}},
	\kbox{\underline\psi_{p\A}},\kbox{\rho}, \kbox{\underline\mu})$ as chosen
	solves~\eqref{eq:def:pressureprimals}. This finishes the proof.
\end{proof}

Analogously to Lemma~\ref{lem:bar:f}, we show that the operator $\bar F_\Pi$ corresponds to taking the maximum in $\widetilde{\mathbf W}_\Pi$. Before we
give a proof, we introduce some useful notation by collecting local
contributions to global matrices and vectors:
\begin{itemize}
		\item The matrix $B$ is a block row matrix containing the corresponding
		patch-local contributions, like $B:=\begin{pmatrix}
			B^{(1)} & \cdots & B^{(K)}
		\end{pmatrix}$.
		\item The matrices $\Psi_{\Gamma\C}$ and $\Psi_{\Gamma\A}$ are block
		column matrices containing the corresponding
		patch-local contributions.
		\item The matrices $S_A$, $S_K$, $D_\Gamma$ and $\scaling$ are block diagonal matrices containing
		the corresponding patch-local contributions.
		\item The vectors, like $\underline{\mathbf w}$, are the corresponding
		block vectors containing the patch-local contributions.
\end{itemize}
\begin{lemma}
	\label{lem:bar:fpi}
	The matrix $\bar F_\Pi$ is symmetric and positive semidefinite and 
	satisfies
		\[
				\| \underline \lambda \|_{\bar F_\Pi}
				=
				\sup_{\mathbf w \in \widetilde{\mathbf W}_\Pi}
				\frac{(B^\top \underline \lambda,\underline{\mathbf w})_{\ell^2}}
				{\|\underline{\mathbf w}\|_{S_A}}.
		\]
\end{lemma}
\begin{proof}
	Recall that
	\begin{align*}
			\bar F_\Pi
			&= 
			\bar B_{\Pi} \bar A_{\Pi}^{-1} \bar B_{\Pi}^\top
			=
			\begin{pmatrix}
			B_{\Pi} & 0
			\end{pmatrix}
			\begin{pmatrix}
			  A_\Pi &C_\Pi^\top\\
  			  C_\Pi & 0
			\end{pmatrix}^{-1}
			\begin{pmatrix}
			B_{\Pi}^\top \\ 0
			\end{pmatrix}
	\end{align*}
	with
	\[
			A_\Pi = \sum_{k=1}^K \kbox{\Psi}^\top \kbox{A} \kbox{\Psi}
			\quad\text{and}\quad
			B_\Pi = \sum_{k=1}^K \kbox{B} \kbox{\Psi}.
	\]
	We decompose the basis for the primal space into basis functions corresponding to the continuity conditions ($\kbox{\Psi_\C}$) and basis functions corresponding to the averaging conditions ($\kbox{\Psi_\A}$):
	\[
		\kbox{\Psi}
		=
		\begin{pmatrix}
				\kbox{\Psi_\C}&\kbox{\Psi_\A}
		\end{pmatrix},
		\quad\mbox{where}\quad
		\kbox{\Psi_\C}:=\begin{pmatrix}\kbox{\Psi_{\Gamma\C}}\\\kbox{\Psi_{\I\C}}\\\kbox{\Psi_{p\C}}\end{pmatrix}
		\quad\mbox{and}\quad 
		\kbox{\Psi_\A}:=\begin{pmatrix}\kbox{\Psi_{\Gamma\A}}\\\kbox{\Psi_{\I\A}}\\\kbox{\Psi_{p\A}}\end{pmatrix}.
	\]
	From~\eqref{eq:basisdef}, we have
	\begin{equation}\label{eq:cbasisharm}
			\begin{pmatrix}
				\kbox{\Psi_{\I\C}}\\
				\kbox{\Psi_{p\C}}\\
				\rho_\C
			\end{pmatrix}
			=
			-
			\begin{pmatrix}
				  \kbox{K_{\I\I}}      & \kbox{D_{\I}}^\top    &0\\
				  \kbox{D_\I}          & 0  & \kbox{C_\A}^\top \\
				  0                    & \kbox{C_\A}    &0
			\end{pmatrix}^{-1}
			\begin{pmatrix}
				\kbox{K_{\I\Gamma}}  \\ \kbox{D_{\Gamma}} \\ 0
			\end{pmatrix}
			\kbox{\Psi_{\Gamma\C}}.
	\end{equation}
	Using $\kbox{C_\A}\kbox{\Psi_{p\C}}=0$, \eqref{eq:cbasisharm} and~\eqref{eq:sgammadef}, we have
	\begin{equation}\label{eq:apiv:spd}
	\begin{aligned}
			\kbox{\Psi_\C}^\top \kbox{A} \kbox{\Psi_\C}
			&=
			\begin{pmatrix}
				\kbox{\Psi_{\Gamma\C}}\\
				\kbox{\Psi_{\I\C}}\\
				\kbox{\Psi_{p\C}}
			\end{pmatrix}^\top
			\begin{pmatrix}
				  \kbox{K_{\Gamma\Gamma}} &\kbox{K_{\Gamma\I}}  & \kbox{D_{\Gamma}}^\top\\
				  \kbox{K_{\I\Gamma}}     &\kbox{K_{\I\I}}      & \kbox{D_{\I}}^\top    \\
				  \kbox{D_\Gamma}         &\kbox{D_\I}          & 0              \\
			\end{pmatrix}		
			\begin{pmatrix}
				\kbox{\Psi_{\Gamma\C}}\\
				\kbox{\Psi_{\I\C}}\\
				\kbox{\Psi_{p\C}}
			\end{pmatrix}
			\\
			&=
			\begin{pmatrix}
				\kbox{\Psi_{\Gamma\C}}\\
				\kbox{\Psi_{\I\C}}\\
				\kbox{\Psi_{p\C}}\\
				\rho_\C
			\end{pmatrix}^\top
			\begin{pmatrix}
				  \kbox{K_{\Gamma\Gamma}} &\kbox{K_{\Gamma\I}}  & \kbox{D_{\Gamma}}^\top&0\\
				  \kbox{K_{\I\Gamma}}     &\kbox{K_{\I\I}}      & \kbox{D_{\I}}^\top    &0\\
				  \kbox{D_\Gamma}         &\kbox{D_\I}          & 0  & \kbox{C_\A}^\top \\
				  0                       &0                    & \kbox{C_\A}    &0
			\end{pmatrix}		
			\begin{pmatrix}
				\kbox{\Psi_{\Gamma\C}}\\
				\kbox{\Psi_{\I\C}}\\
				\kbox{\Psi_{p\C}}\\
				\rho_\C
			\end{pmatrix}
			\\&=
			\kbox{\Psi_{\Gamma\C}}^\top \kbox{S_A} \kbox{\Psi_{\Gamma\C}} \ge 0.
	\end{aligned}
	\end{equation}
	Using Lemma~\ref{lemma:abasis}, we have
	\begin{equation}\label{eq:apip:spd}
			\kbox{\Psi_\A}^\top \kbox{A} \kbox{\Psi_\A}
			=
			\begin{pmatrix}
				0\\
				0\\
				\underline 1 \kbox{R_\A}
			\end{pmatrix}^\top
			\begin{pmatrix}
				  \kbox{K_{\Gamma\Gamma}} &\kbox{K_{\Gamma\I}}  & \kbox{D_{\Gamma}}^\top\\
				  \kbox{K_{\I\Gamma}}     &\kbox{K_{\I\I}}      & \kbox{D_{\I}}^\top    \\
				  \kbox{D_\Gamma}         &\kbox{D_\I}          & 0              \\
			\end{pmatrix}
			\begin{pmatrix}
				0\\
				0\\
				\underline 1 \kbox{R_\A}
			\end{pmatrix}
			=
			0
	\end{equation}
	and
	\begin{equation}\label{eq:amixed}
	\begin{aligned}
			\kbox{\Psi_\A}^\top \kbox{A} \kbox{\Psi_\C}
			&=
			\begin{pmatrix}
				0\\
				0\\
				\underline 1 \kbox{R_\A}
			\end{pmatrix}^\top
			\begin{pmatrix}
				  \kbox{K_{\Gamma\Gamma}} &\kbox{K_{\Gamma\I}}  & \kbox{D_{\Gamma}}^\top\\
				  \kbox{K_{\I\Gamma}}     &\kbox{K_{\I\I}}      & \kbox{D_{\I}}^\top    \\
				  \kbox{D_\Gamma}         &\kbox{D_\I}          & 0              \\
			\end{pmatrix}
			\begin{pmatrix}
				\kbox{\Psi_{\Gamma\C}}\\
				\kbox{\Psi_{\I\C}}\\
				\kbox{\Psi_{p\C}}
			\end{pmatrix}
			\\&=
			\kbox{R_\A}^\top
			\underline 1^\top \kbox{D_\Gamma} \kbox{\Psi_{\Gamma\C}}
	\end{aligned}
	\end{equation}	
	since $\underline 1^\top \kbox{D_\I} \kbox{\Psi_{\I\C}}=0$.
	Using \eqref{eq:apiv:spd}, \eqref{eq:apip:spd} and \eqref{eq:amixed}, we obtain
	\begin{align*}
		\bar F_\Pi
		& =
		\begin{pmatrix}
			B_{\Pi} & 0
			\end{pmatrix}
			\begin{pmatrix}
			  A_\Pi &C_\Pi^\top\\
  			C_\Pi & 0
			\end{pmatrix}^{-1}
			\begin{pmatrix}
			B_{\Pi}^\top \\ 0
			\end{pmatrix}
		\\
		& =
			\begin{pmatrix}
			B \Psi_{\Gamma\C} & 0 & 0
			\end{pmatrix}
			\begin{pmatrix}
			    \Psi_{\Gamma\C}^\top S_A \Psi_{\Gamma\C} 
				& \Psi_{\Gamma\C}^\top Z 
				& 0 \\		
			   Z^\top \Psi_{\Gamma\C} 
				& 0
				& C_\Pi^\top \\
				0
				& C_\Pi
				& 0
			\end{pmatrix}^{-1}
			\begin{pmatrix}
			\Psi_{\Gamma\C}^\top B^\top \\ 0 \\ 0
			\end{pmatrix},
	\end{align*}
	where $Z$ is a block-diagonal matrix containing
	$\kkbox{D_\Gamma}{1}^\top\underline 1, \ldots, \kkbox{D_\Gamma}{K}^\top\underline 1$.
	Using Lemma~\ref{lem:supmat}, we have
	\[
		\|\underline \lambda\|_{\bar F_\Pi}
		=
		\sup_{\underline{\mathbf w}_{\C} \in \underline{\mathbf W_{\C}}}
		\frac{ (B^\top \underline \lambda, \Psi_{\Gamma\C} \underline{\mathbf w}_{\C})_{\ell^2} }
		{ \|\Psi_{\Gamma\C} \underline{\mathbf w}_{\C}\|_{S_A}},
	\]
	where $\underline{\mathbf W}_{\C}:=\{ \underline{\mathbf w}_{\C} :
	C_\Pi \underline \mu = 0 \Rightarrow 
	\underline\mu^\top Z^\top \Psi_{\Gamma\C}\underline{\mathbf w}_{\C}=0\}$. Define
	$\underline {\mathbf w}:=\Psi_{\Gamma\C} \underline{\mathbf w}_{\C}$  and let
	$\mathbf w=(\mathbf w^{(1)},\cdots,\mathbf w^{(K)})\in \widetilde{\mathbf W}$ be the
	function associated to the coefficient vector $\underline{\mathbf w}$.
	Observe that the condition
	\[
			C_\Pi \underline \mu =0 \Rightarrow 
					\underline\mu^\top Z^\top \underline{\mathbf w}
					=\sum_{k=1}^K \mu_k \underline 1^\top \kbox{D_{\Gamma}}
					\kbox{\underline{\mathbf w}}
					=0
	\]
	translates to
	\[
			\sum_{k=1}^K \mu_k = 0
			\Rightarrow
			\sum_{k=1}^K
			\mu_k \int_{\partial\Omega^{(k)}} \kbox{\mathbf w} \cdot \kbox{\mathbf n} \, \mathrm d s = 0.
	\]
	Since functions in $\widetilde{\mathbf W}$ satisfy homogeneous Dirichlet boundary conditions and since they satisfy a continuity condition on the averages of the normal components of the velocity,
	 we also have
	\[
			0 = 
			\int_{\partial\Omega} \mathbf w \cdot \mathbf n \, \mathrm d s
			= \sum_{k=1}^K
			\int_{\partial\Omega^{(k)}} \mathbf w^{(k)} \cdot \mathbf n^{(k)} \, \mathrm d s.
	\]
	By combining these results, we obtain that
	\[
			\int_{\partial\Omega^{(k)}} \kbox{\mathbf w} \cdot \kbox{\mathbf n} \, \mathrm d s = 0
			\qquad\foralls k=1,\ldots,K.
	\]
	Since $\mathbf w$ is in the image space of the primal basis
	functions, we also know that $\mathbf w$ satisfies the remaining conditions in the
	definition of $\widetilde{\mathbf W}$ and that it is orthogonal to
	$\widetilde{\mathbf W}_\Delta$. This shows $\mathbf w\in \widetilde{\mathbf W}_\Pi$.
	The reverse direction, i.e., that for each $\mathbf w\in \widetilde{\mathbf W}_\Pi$,
	there is some $\underline{\mathbf w}_{\C} \in \underline{\mathbf W}_{\C}$ with $\underline{\mathbf w}:=\Psi_{\Gamma\C} \underline{\mathbf w}_{\C}$
	is straight forward. So, we obtain
	\[
		\|\underline \lambda\|_{\bar F_\Pi}
		=
		\sup_{\mathbf w \in \widetilde{\mathbf W}_\Pi}
		\frac{ (B^\top \underline \lambda, \underline{\mathbf w})_{\ell^2} }
		{ \|\underline{\mathbf w}\|_{S_A}},
	\]
	which is what we wanted to show. This representation immediately shows that the symmetric matrix
	$\bar F_\Pi$ is positive semidefinite.
\end{proof}

From the Lemmas~\ref{lem:bar:f} and~\ref{lem:bar:fpi}, we immediately obtain that
the matrix $\bar F$ is symmetric positive semidefinite and that the following result holds.
\begin{lemma}
The identity
\begin{equation}\label{eq:barF}
		\| \underline \lambda \|_{\bar F}
		=
		\sup_{\mathbf w \in \widetilde{\mathbf W}}
		\frac{(B^\top \underline \lambda,\underline{\mathbf w})_{\ell^2}}{\|\underline{\mathbf w}\|_{S_A}}
\end{equation}
holds for all $\underline\lambda$.
\end{lemma}
\begin{proof}
Using orthogonality, the Cauchy-Schwarz inequality and Lemmas~\ref{lem:bar:f} and~\ref{lem:bar:fpi}, we have
\begin{align*}
	&\sup_{\mathbf w \in \widetilde{\mathbf W}}
		\frac{(B^\top  \underline \lambda,\underline{\mathbf w})_{\ell^2}^2}{\|\underline{\mathbf w}\|_{S_A}^2}
	=
	\sup_{\sum_k \mathbf w_\Delta^{(k)} + \mathbf w_\Pi \in \widetilde{\mathbf W}}
	\frac{( \underline \lambda,\sum_{k=1}^K  \kbox{B}\underline{\mathbf w}_\Delta^{(k)} + B \underline{\mathbf w}_\Pi )_{\ell^2}^2}
	{\sum_{k=1}^K \|\underline{\mathbf w}_\Delta^{(k)}\|_{\kbox{S_A}}^2+
	\|\underline{\mathbf w}_\Pi\|_{S_A}^2}
	\\&\quad =
	\sup_{\sum_k \mathbf w_\Delta^{(k)} + \mathbf w_\Pi \in \widetilde{\mathbf W}}
	\frac{\left(
	\sum_{k=1}^K \frac{( \underline \lambda, \kbox{B } \underline{\mathbf w}_\Delta^{(k)} )_{\ell^2}}{\|\underline{\mathbf w}_\Delta^{(k)}\|_{\kbox{S_A}}} \|\underline{\mathbf w}_\Delta^{(k)}\|_{\kbox{S_A}}
	+
	\frac{( \underline \lambda, B \underline{\mathbf w}_\Pi )_{\ell^2}}{\|\underline{\mathbf w}_\Pi\|_{S_A}} \|\underline{\mathbf w}_\Pi\|_{S_A}
	\right)^2}
	{\sum_{k=1}^K \|\underline{\mathbf w}_\Delta^{(k)}\|_{\kbox{S_A}}^2+
	\|\underline{\mathbf w}_\Pi\|_{S_A}^2}
	\\&\quad\le
	\sup_{\sum_k \mathbf w_\Delta^{(k)} + \mathbf w_\Pi \in \widetilde{\mathbf W}}
	\left(
	\sum_{k=1}^K 
	\frac{( \underline \lambda, \kbox{B } \underline{\mathbf w}_\Delta^{(k)} )_{\ell^2}^2}
	{\|\underline{\mathbf w}_\Delta^{(k)}\|_{\kbox{S_A}}^2}
	+
	\frac{( \underline \lambda, B \underline{\mathbf w}_\Pi )_{\ell^2}^2}
	{\|\underline{\mathbf w}_\Pi\|_{S_A}^2}
	\right)
	=
	\|\underline{\lambda}\|_{\bar F}^2
	.
\end{align*}
The Cauchy-Schwarz inequality is satisfied with equality if the corresponding terms are equal, i.e.,
\[
		\frac{( \underline \lambda, \kbox{B} \underline{\mathbf w}_\Delta^{(k)} )_{\ell^2}}
	{\|\underline{\mathbf w}_\Delta^{(k)}\|_{\kbox{S}}}
	= \|\underline{\mathbf w}_\Delta^{(k)}\|_{\kbox{S}}
	\quad\mbox{and}\quad
	\frac{( \underline \lambda, B \underline{\mathbf w}_\Pi )_{\ell^2}}
	{\|\underline{\mathbf w}_\Pi\|_{S_A}}
	=
	\|\underline{\mathbf w}_\Pi\|_{S_A}.
\]
Let $\mathbf W^*\subset \widetilde{\mathbf W}$ be the subset of functions that satisfy these
conditions. Due to scaling invariance and the fact that the Cauchy-Schwarz inequality is
satisfied with equality and $\mathbf W^*\subset \widetilde{\mathbf W}$, we have
\begin{align*}
	\|\underline{\lambda}\|_{\bar F}^2
	&=
	\sup_{\sum_k \mathbf w_\Delta^{(k)} + \mathbf w_\Pi \in \mathbf W^*}
	\left(
	\sum_{k=1}^K 
	\frac{( \underline \lambda, \kbox{B} \underline{\mathbf w}_\Delta^{(k)} )_{\ell^2}^2}
	{\|\underline{\mathbf w}_\Delta^{(k)}\|_{\kbox{S_A}}^2}
	+
	\frac{( \underline \lambda, B \underline{\mathbf w}_\Pi)_{\ell^2}^2}
	{\|\underline{\mathbf w}_\Pi\|_{S_A}^2}
	\right)
	\\&=
	\sup_{\mathbf w \in \mathbf W^*}
		\frac{(B^\top  \underline \lambda,\underline{\mathbf w})_{\ell^2}^2}{\|\underline{\mathbf w}\|_{S_A}^2}
	\le
	\sup_{\mathbf w \in \widetilde{\mathbf W}}
		\frac{(B^\top  \underline \lambda,\underline{\mathbf w})_{\ell^2}^2}{\|\underline{\mathbf w}\|_{S_A}^2},
\end{align*}
which finishes the proof.
\end{proof}

\begin{lemma}\label{lem:BBt}
	Let $\underline{\mathbf w} = \mathcal D^{-1} B^\top B \underline{\mathbf v}$, where 
	$\underline{\mathbf w}$ and $\underline{\mathbf v}$ are the coefficient representations of
	$\mathbf w\in {\mathbf W}$ and $\mathbf v\in \widetilde{\mathbf W}$, respectively.
	Then, we have
	\[
			\mathbf w^{(k)}|_{\Gamma^{(k,\ell)}}
			=
			\mathbf v^{(k)}|_{\Gamma^{(k,\ell)}}
			-
			\mathbf v^{(\ell)}|_{\Gamma^{(k,\ell)}}.
	\]
\end{lemma}
This lemma is standard, for a proof, see, e.g., Lemma~4.16 in Ref.~\refcite{SchneckenleitnerTakacs:2020}.

\begin{lemma}\label{lem:Wtildepreserv}
	Let $\underline{\mathbf w} = \scaling^{-1} B^\top B \underline{\mathbf v}$, where 
	$\underline{\mathbf w}$ and $\underline{\mathbf v}$ are the coefficient representations of
	$\mathbf w\in {\mathbf W}$ and $\mathbf v\in \widetilde{\mathbf W}$, respectively.
	We have $\mathbf w\in \widetilde{\mathbf W}$.
\end{lemma}
\begin{proof}
	Since the basis functions on the vertices are not affected by the constraints, functions
	that are represented by coefficient vectors in the image of $B$ vanish on the vertices.
	This guarantees continuity at the vertices.
	Lemma~\ref{lem:BBt} and the continuity of the integrals of the normal components of $\mathbf v$ imply that the corresponding integrals of $\mathbf w$ vanish.
	This finishes the proof.
\end{proof}

\begin{lemma}\label{lem:BdBtB}
	The identity $B \scaling^{-1} B^\top B=B$ holds.
\end{lemma}
Since we exclude the corners, this statement is standard, see,
	e.g., Ref.~\refcite{MandelDohrmannTezaur:2005a}.

\begin{lemma}\label{lem:BdtB}
	For all $\mathbf v \in \widetilde{\mathbf W}$ with coefficient representation $\underline{\mathbf v}$, the estimate
	\[
		\| \scaling^{-1} B^\top B \underline{\mathbf v} \|_{S_K}^2
		\le
		C\, \degree \left(1+\log \degree+\max_{k=1,\ldots,K} \log\frac{H_k}{h_k}\right)^2
		\| \underline{\mathbf v}  \|_{S_K}^2
	\]
	holds,
	where $C$ only depends on the constants from the
	Assumptions~\ref{ass:nabla} and~\ref{ass:quasiuniform}.
\end{lemma}
\begin{proof}
	Within this proof, we write $a\lesssim b$ if there is a constant $c$ that
	only depends on the constants from the
	Assumptions~\ref{ass:nabla} and~\ref{ass:quasiuniform} such that
	$a\le c\,b$.
	Let $\mathbf v\in\widetilde{\mathbf W}$ and $\mathbf w\in\mathbf W$ with
	coefficient representations such that
	$\underline{\mathbf w} = \scaling^{-1} B^\top B \underline{\mathbf v}$.
	Using Theorem~4.2 from Ref.~\refcite{SchneckenleitnerTakacs:2020} (which depends on the constants
	from Assumptions~\ref{ass:nabla} and~\ref{ass:quasiuniform}), we obtain
	\[
		\| \scaling^{-1} B^\top B \underline{\mathbf v} \|_{S_K}^2
		=
		\| \underline{\mathbf w} \|_{S_K}^2
		\lesssim
		\sum_{k=1}^K
		| \mathcal H_k \mathbf w^{(k)} |_{H^1(\Omega^{(k)})}^2
		\lesssim
		\degree
		\sum_{k=1}^K
		| \mathbf w^{(k)} |_{H^{1/2}(\partial\Omega^{(k)})}^2,
	\]
	where we apply Theorem~4.2 from Ref.~\refcite{SchneckenleitnerTakacs:2020} to both components
	of the velocity variable separately.
	By applying Lemma~4.15 from Ref.~\refcite{SchneckenleitnerTakacs:2020} (which depends on the constants
	from Assumptions~\ref{ass:nabla} and~\ref{ass:quasiuniform}) to both velocity components
	separately, we further obtain
	\[
		\| \scaling^{-1} B^\top B \underline{\mathbf v} \|_{S_K}^2
		\lesssim
		\degree
		\sum_{k=1}^K
		\sum_{\ell\in\mathcal N_\Gamma(k)}
		\big(
		| \mathbf w^{(k)} |_{H^{1/2}(\Gamma^{(k,\ell)})}^2
		+
		\Lambda
		| \mathbf w^{(k)} |_{L_0^\infty(\Gamma^{(k,\ell)})}^2
		\big),
	\]
	where $\Lambda:=1+\log \degree+\max_{k=1,\ldots,K} \log\frac{H_k}{h_k}$
	and
	\[
	|v|_{L_0^\infty(\Gamma^{(k,\ell)})}:=\inf_{q\in\mathbb R}
	\|v-q\|_{L^\infty(\Gamma^{(k,\ell)})}.
	\]
	Using Lemma~\ref{lem:BBt} and the triangle inequality, we further obtain
	\begin{align*}
		&\| \scaling^{-1} B^\top B \underline{\mathbf v} \|_{S_K}^2
		\\&\quad\lesssim
		\degree
		\sum_{k=1}^K
		\sum_{\ell\in\mathcal N_\Gamma(k)}
		\big(
		| \mathbf v^{(k)} - \mathbf v^{(\ell)} |_{H^{1/2}(\Gamma^{(k,\ell)})}^2
		+
		\Lambda
		| \mathbf v^{(k)} - \mathbf v^{(\ell)} |_{L_0^\infty(\Gamma^{(k,\ell)})}^2
		\big)
		\\&\quad\lesssim
		\degree
		\sum_{k=1}^K
		\big(
		| \mathbf v^{(k)} |_{H^{1/2}(\Gamma^{(k,\ell)})}^2
		+
		\Lambda
		| \mathbf v^{(k)} |_{L_0^\infty(\Gamma^{(k,\ell)})}^2
		\big).
	\end{align*}
	By applying Theorem~4.2 and Lemma~4.14 from Ref.~\refcite{SchneckenleitnerTakacs:2020}
	(which depend on the constants
	from Assumptions~\ref{ass:nabla} and~\ref{ass:quasiuniform}) again to both velocity components,
	we arrive at
	\begin{align*}
		\| \scaling^{-1} B^\top B \underline{\mathbf v} \|_{S_K}^2
		\lesssim
		\degree
		\sum_{k=1}^K
		\big(
		| \mathcal H_k \mathbf v^{(k)} |_{H^{1}(\Omega^{(k)})}^2
		+
		\Lambda^2
		\inf_{q\in\mathbb R}
		\| \mathcal H_k \mathbf v^{(k)} - q \|_{H^{1}(\Omega^{(k)})}^2
		\big).
	\end{align*}
	Using a Poincaré inequality~\eqref{eq:poincare}, we have
	\begin{align*}
		\| \scaling^{-1} B^\top B \underline{\mathbf v} \|_{S_K}^2
		\lesssim
		\degree
		\sum_{k=1}^K
		\big(
		| \mathcal H_k \mathbf v^{(k)} |_{H^{1}(\Omega^{(k)})}^2
		+
		\Lambda^2
		| \mathcal H_k \mathbf v^{(k)} |_{H^{1}(\Omega^{(k)})}^2
		\big),
	\end{align*}	
	from which the desired result follows.	
\end{proof}

If standard Krylov space methods are applied to the singular matrix $F$,
preconditioned with a non-singular preconditioner $M_{\mathrm{sD}}$,
all iterations live in the corresponding factor space. The convergence
behavior is dictated by the essential condition number of
$M_{\mathrm{sD}}F$, cf. Remark~23 in Ref.~\refcite{MandelDohrmannTezaur:2005a}.
The essential condition number for a positive semidefinite matrix
is the ratio between the largest eigenvalue and the smallest positive
eigenvalue.

\begin{theorem}\label{thrm:fin}
	Provided that the IETI-DP solver is set up as outlined in the previous
	section, the condition number of the preconditioned system satisfies
	\[
		\kappaess(M_{\mathrm{sD}} F) \le C\,
		 \degree \left(1+\log \degree+\max_{k=1,\ldots,K} \log\frac{H_k}{h_k}\right)^2\,
		\left(\max_{k=1,\ldots,K}\frac{\delta}{\beta_k}\right),
	\]
	where $\delta$ is as in~\eqref{eq:boundedness},
	$\beta_k$ is as in~\eqref{eq:infsuplocal}
	and and $C$ is a constant that
	only depends on the constants from the
	Assumptions~\ref{ass:nabla} and~\ref{ass:quasiuniform}.
\end{theorem}
\begin{proof}
Within this proof, we write $a\lesssim b$ if there is a constant $c$ that
	only depends on the constants from the
	Assumptions~\ref{ass:nabla} and~\ref{ass:quasiuniform} such that
	$a\le c\,b$.
For an upper bound, we have using Lemma~\ref{lem:h1like},~\eqref{eq:barF},
Lemma~\ref{lem:BdtB} and the fact that $B$ has full rank that
\begin{align*}
&\sqrt{\underline \lambda^\top \bar F \underline \lambda}
	=
	\sup_{\mathbf w \in \widetilde{\mathbf W}\backslash\{0\}}
	\frac{(B^\top \underline \lambda, \underline{\mathbf w})_{\ell^2} }{\|\underline{\mathbf w}\|_{S_A}}
	\le
	\sup_{\mathbf w \in \mathbf W\backslash\{0\}}
	\frac{(B^\top \underline \lambda, \underline{\mathbf w})_{\ell^2} }{\|\underline{\mathbf w}\|_{S_K}}
	=
	\sup_{\mathbf w \in \mathbf W\backslash\mathrm{Ker}B}
	\frac{( \underline \lambda, B \underline{\mathbf w})_{\ell^2} }{\|\underline{\mathbf w}\|_{S_K}}\\
	&  \quad \lesssim \omega
	\sup_{\mathbf w \in \mathbf W\backslash\mathrm{Ker}B}
	\frac{( \underline \lambda, B \underline{\mathbf w})_{\ell^2} }{\|\scaling^{-1} B^\top B \underline{\mathbf w}\|_{S_K}}
	\le
	\omega \sup_{\underline \mu\in L \backslash\{0\}}
	\frac{(\underline \lambda, \underline \mu)_{\ell^2} }{\|\scaling^{-1} B^\top  \underline \mu\|_{S_K}}
	\\&\quad=
	\omega \sup_{\underline \mu\in L\backslash\{0\}}
	\frac{(\underline \lambda, \underline \mu)_{\ell^2} }{\| \underline \mu\|_{M_{\mathrm{sD}}}}
	= \omega \; \sqrt{\underline \lambda^\top M_{\mathrm{sD}}^{-1} \underline \lambda},
\end{align*}
where $L$ is a vector space of the corresponding dimension and
$\omega^2:=\degree ( 1+\log \degree + \max_k \log \frac{H_k}{h_k})^2$.
This provides an upper bound for the eigenvalues of $M_{\mathrm{sD}}^{-1} \bar F$.

Next, we estimate the smallest non-zero eigenvalue. Consider the
generalized eigenvalue problem
\[
		\bar F\underline \lambda = \mu M_{\mathrm{sD}}^{-1} \underline\lambda.
\]
We are interested in the smallest non-zero eigenvalue $\mu$. Define
\[
		L_0 := \{
							\underline\lambda_0 : (\underline\lambda_0,
												B\, \underline{\mathbf w})_{\ell^2}=0
											\;\forall\, \mathbf w\in\widetilde{\mathbf W}
					 \}
\]
and
\[
		L_1 := \{
							\underline\lambda_1 : \underline\lambda_1
												= M_{\mathrm{sD}}B\, \underline{\mathbf w}
														\mbox{ with }
														\mathbf w \in \widetilde{\mathbf W}
					 \}
\]
and observe that $(\underline\lambda_0, \underline\lambda_1)_{
M_{\mathrm{sD}}^{-1}}=0$ and $\bar F\underline\lambda_0 =0$
for all $\underline \lambda_0\in L_0$ and $\underline \lambda_1 \in L_1$. This means that
all $\underline \lambda_0\in L_0$ are eigenvectors with eigenvalue $0$. Since all eigenvectors with non-zero eigenvalue are $M_{\mathrm{sD}}^{-1}$-orthogonal to the eigenvectors with eigenvalue $0$, these eigenvectors have to be in $L_1$. Using $\underline \lambda_1\in L_1\backslash\{0\}$ and Lemma~\ref{lem:BdBtB},
we have 
\begin{align*}
	\sqrt{\underline \lambda_1^\top M_{\mathrm{sD}}^{-1} \underline \lambda_1}
	&=
	\frac{
	(M_{\mathrm{sD}}^{-1} \underline \lambda_1, \underline \lambda_1)_{\ell^2}
	}{
	\|M_{\mathrm{sD}}^{-1} \underline \lambda_1\|_{M_{\mathrm{sD}}}
	}
	\le
	\sup_{\mathbf w \in \widetilde{\mathbf W}\backslash\mathrm{Ker}B}
	\frac{
	(B \underline{\mathbf w},\underline\lambda_1)_{\ell^2}
	}{
	\|B \underline{\mathbf w}\|_{M_{\mathrm{sD}}}
	}
	\\&=
	\sup_{\mathbf w \in \widetilde{\mathbf W}\backslash\mathrm{Ker}B}
	\frac{
	(B \scaling^{-1} B^\top B \underline{\mathbf w},\underline\lambda_1)_{\ell^2}
	}{
	\|\scaling^{-1} B^\top B \underline{\mathbf w}\|_{S_{K}}
	}.
\end{align*}
Using Lemma~\ref{lem:Wtildepreserv}, we know that
$\underline{\mathbf v}:=\scaling^{-1} B^\top B \underline{\mathbf w} \in \widetilde{\mathbf W}$,
so using Lemma~\ref{lem:h1like} we further obtain  
\begin{align*}
	\sqrt{\underline \lambda_1^\top M_{\mathrm{sD}}^{-1} \underline \lambda_1}
	&\le
	\sup_{\mathbf v \in \widetilde{\mathbf W}\backslash\{0\}}
	\frac{
	(B \underline{\mathbf v},\underline\lambda_1)_{\ell^2}
	}{
	\|\underline{\mathbf v}\|_{S_{K}}
	}
	\lesssim
  \left(\max_{k=1,\ldots,K} \frac{\delta}{\beta_k}\right)^{-1}
	\sup_{\mathbf v \in \widetilde{\mathbf W}\backslash\{0\}}
	\frac{
	(B  \underline{\mathbf v},\underline\lambda_1)_{\ell^2}
	}{
	\|\underline{\mathbf v}\|_{S_A}
	}
	\\
	&=
	\left(\max_{k=1,\ldots,K} \frac{\delta}{\beta_k}\right)^{-1}
	\sqrt{\underline\lambda_1^\top \bar F\underline \lambda_1},
\end{align*}
which provides a lower bound for the positive eigenvalues.
Having these eigenvalue bounds, we immediately obtain
the desired bound on the essential condition number.
\end{proof}

\section{Numerical results}
\label{sec:6}

In this section, we present numerical results that illustrate the efficiency of the proposed IETI-DP solver.
In Subsection~\ref{subsec:6:1}, we present results for domains that have been previously
considered in IgA and which are fully covered by the presented theory. In Subsection~\ref{subsec:6:2},
we present a more physical test example using boundary conditions that go beyond the model problem considered for the theory.

\subsection{Results for quarter annulus and Yeti-footprint}
\label{subsec:6:1}

We consider the Stokes problem~\eqref{eq:stokes} with the right-hand-side function
\begin{align*}
  \mathbf{f}(x,y) = (-\pi\cos(\pi x)-2 \pi^2\sin(\pi x) \cos(\pi y),\, 2\pi^2\cos(\pi x) \sin(\pi y))^\top
\end{align*}
and the inhomogeneous Dirichlet boundary conditions
\begin{align*}
  \mathbf{u}(x,y) &= (-\sin(\pi x)\cos(\pi y),\, \cos(\pi x)\sin(\pi y))^\top \quad \text{for} \quad (x,y)\in\partial\Omega.
\end{align*}
We solve this problem on two computational domains: a B-spline approximation of a quarter annulus with 64 patches and the Yeti-footprint, we where we split the patches uniformly such that the domain has 84 patches (rather than 21), see Figure~\ref{fig:domain}.
\begin{figure}[htb]
	\centering
	\includegraphics[height=4cm]{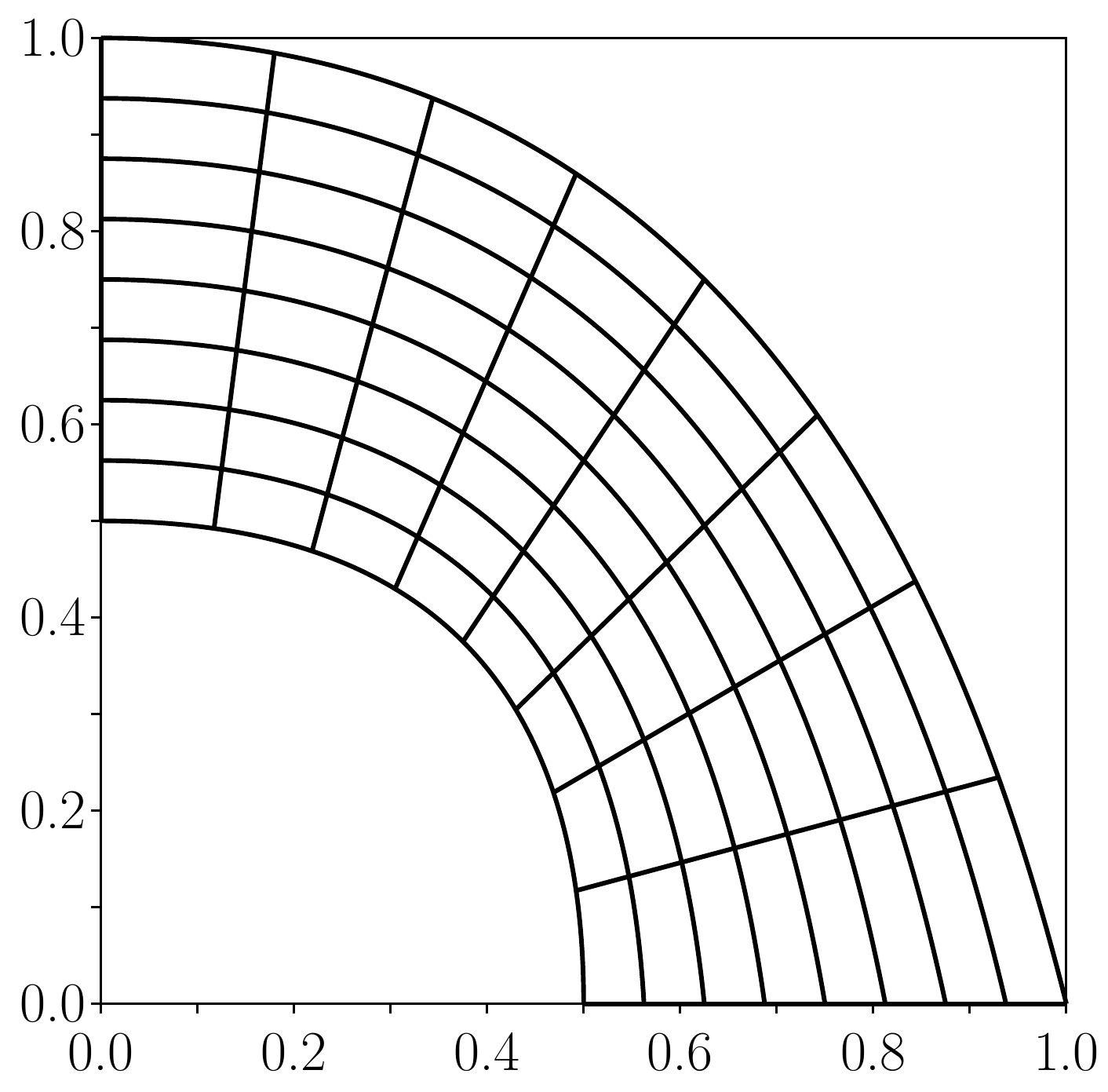}
    \qquad
	\includegraphics[height=4cm]{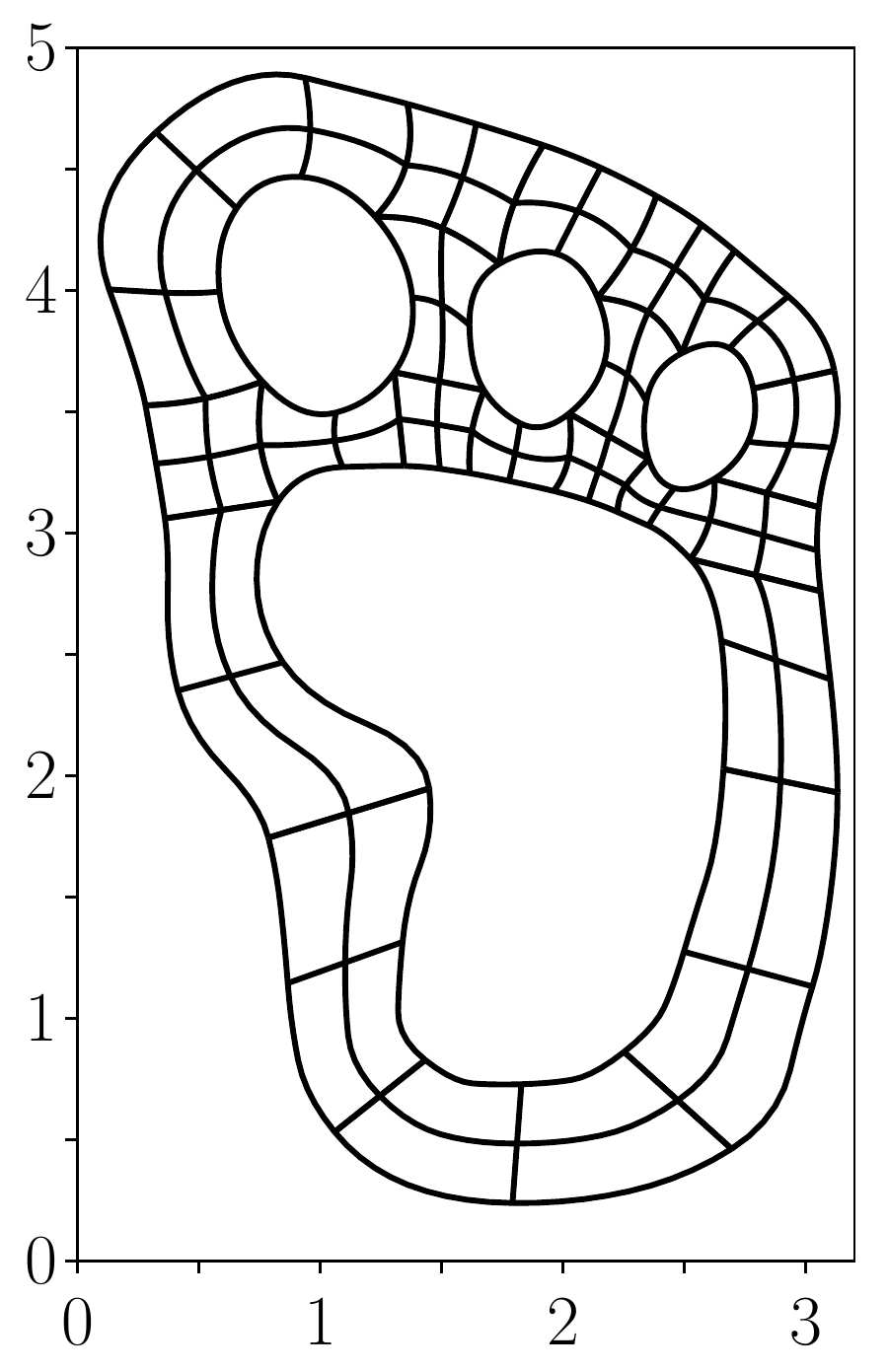}
	\caption{Domains: quarter annulus (left) and Yeti-footprint (right)}
	\label{fig:domain}
\end{figure}
On each patch, we obtain a discretization space by performing $\ell = 1,2,\ldots$ uniform refinement steps. The problem is discretized as outlined in Section~\ref{sec:3} and solved using the IETI-DP method proposed Section~\ref{sec:4}. The conjugate gradient solver is started using a random initial guess and stopped when the Euclidean norm of the residual vector is reduced by a factor of $\epsilon = 10^{-6}$ compared to the Euclidean norm of the initial residual vector. While preforming the conjugate gradient method, we also estimate the condition number based on the underlying Lanczos iteration. The local linear systems are solved using a sparse LU-solver. All experiments have been implemented using the G+Smo library\footnote{\url{https://github.com/gismo/gismo}} and have been performed on the Radon1 cluster\footnote{\url{https://www.ricam.oeaw.ac.at/hpc/}} in Linz.
\begin{table}[th]
\scriptsize
	\newcolumntype{L}[1]{>{\raggedleft\arraybackslash\hspace{-1em}}m{#1}}
	\centering
	\renewcommand{\arraystretch}{1.25}
	\begin{tabular}{l|L{1em}L{1.8em}|L{1em}L{1.8em}|L{1em}L{1.8em}|L{1em}L{1.8em}|L{1em}L{1.8em}}
		\toprule
		\multicolumn{1}{l}{$\ell\;\;\diagdown\;\;\degree$ \hspace{-1.8em}\;}
		& \multicolumn{2}{c|}{2}
		& \multicolumn{2}{c|}{3}
		& \multicolumn{2}{c|}{4}
		& \multicolumn{2}{c|}{5}
		& \multicolumn{2}{c}{6} \\
		& it & $\kappa$
		& it & $\kappa$
		& it & $\kappa$
		& it & $\kappa$
		& it & $\kappa$ \\
		\midrule
		$2$  & $17$ & $7.3$  & $17$ & $8.2$  & $17$ & $8.5$  & $17$ & $9.3$  & $16$ & $9.3 $ \\
		$3$  & $18$ & $8.7$  & $19$ & $9.8$  & $19$ & $10.3$ & $18$ & $10.9$ & $18$ & $11.0$ \\
		$4$  & $20$ & $10.2$ & $20$ & $11.4$ & $20$ & $11.7$ & $20$ & $12.9$ & $19$ & $12.7$ \\
		$5$  & $22$ & $12.7$ & $22$ & $13.8$ & $22$ & $14.3$ & $21$ & $14.7$ & $21$ & $15.6$ \\
		\bottomrule
	\end{tabular}
	\captionof{table}{Iteration counts (it) and condition numbers $\kappa$, quarter annulus
		\label{tab:QuarterAnnulus}}
\end{table}
\begin{table}[th]
\scriptsize
	\newcolumntype{L}[1]{>{\raggedleft\arraybackslash\hspace{-1em}}m{#1}}
	\centering
	\renewcommand{\arraystretch}{1.25}
	\begin{tabular}{l|L{1em}L{1.8em}|L{1em}L{1.8em}|L{1em}L{1.8em}|L{1em}L{1.8em}|L{1em}L{1.8em}}
		\toprule
		\multicolumn{1}{l}{$\ell\;\;\diagdown\;\;\degree$\hspace{-1.8em}\;}
		& \multicolumn{2}{c|}{2}
		& \multicolumn{2}{c|}{3}
		& \multicolumn{2}{c|}{4}
		& \multicolumn{2}{c|}{5}
		& \multicolumn{2}{c}{6} \\
		& it & $\kappa$
		& it & $\kappa$
		& it & $\kappa$
		& it & $\kappa$
		& it & $\kappa$ \\
		\midrule
		$2$  & $16$ & $ 7.9$ & $17$ & $ 8.8$ & $16$ & $ 9.7$ & $16$ & $10.3$ & $16$ & $10.9$ \\
		$3$  & $18$ & $ 9.6$ & $18$ & $10.1$ & $18$ & $11.6$ & $18$ & $12.1$ & $17$ & $12.2$ \\
		$4$  & $20$ & $11.6$ & $20$ & $12.8$ & $20$ & $13.7$ & $19$ & $14.4$ & $19$ & $14.9$ \\
		$5$  & $22$ & $13.7$ & $22$ & $14.9$ & $22$ & $15.9$ & $21$ & $16.6$ & $21$ & $17.4$ \\
		\bottomrule
	\end{tabular}
	\captionof{table}{Iteration counts (it) and condition numbers $\kappa$, Yeti-footprint
		\label{tab:YetiFoot}}
\end{table}

We present the iteration counts and the condition numbers in Tables~\ref{tab:QuarterAnnulus} (Quarter annulus) and~\ref{tab:YetiFoot} (Yeti-footprint). Following from the Table, we see that the iteration counts and the condition numbers grow about linearly in $\ell$, which corresponds to a growth like $\log\frac{H_k}{h_k}$, which is slower than $(\log\frac{H_k}{h_k})^2$, predicted by the theory. The growth in the spline degree parameter~$\degree$ seems to be linear or sublinear. Here, the theory does not tell the complete story since we do not actually know how the inf-sup constant depends on the spline degree. For both domains, the condition numbers and the iterations counts are very satisfactory, keeping in mind the results from Subsection~\ref{subsec:3:5}.

\subsection{Flow through a rectangle with an obstacle}
\label{subsec:6:2}

In this section, we consider a stationary flow through a rectangle with a circular hole.
This domain consists 11 patches, see Figure~\ref{fig:rectangle}. The first four patches, which are adjacent to the circle are parameterized using NURBS. The remaining patches are parameterized using a standard affine mapping, represented as tensor-product B-spline mappings. Even though some patches are parameterized NURBS, we use tensor-product B-splines to set up the discrete function spaces.
Starting from a coarsest level with no interior knots, we perform $\ell= 1,2,3,4,5$ uniform refinements to obtain the grid used for the simulation.

\begin{figure}[htb]
	\centering
	\includegraphics[height=2.175cm]{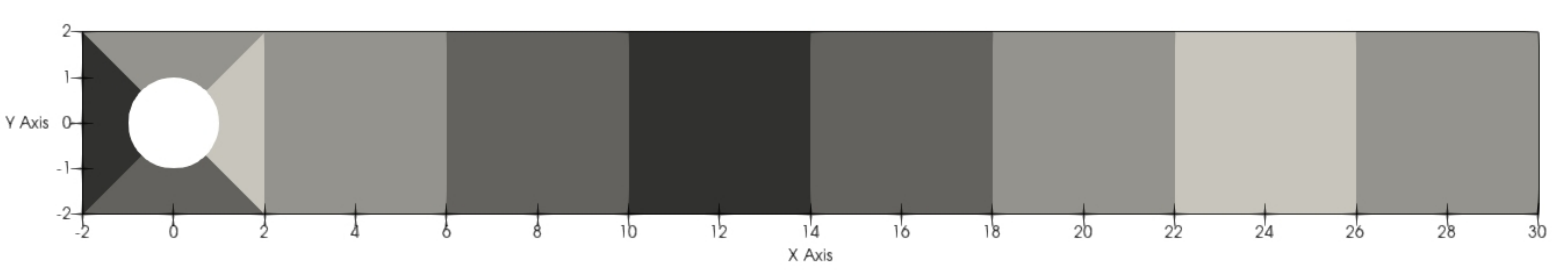}
    \caption{Rectangular domain with a circular hole}
	\label{fig:rectangle}
\end{figure}

The differential equation is set up as follows. We have a zero source term, $\mathbf f=0$, and the boundary conditions are chosen as follows. For $(-2,30)\times\{-2\}$ and $(-2,30)\times\{2\}$, we choose a noslip condition, i.e., $\mathbf{u} = (0,0)$. On $\{-2\}\times(-2,2)$, we use an inlet boundary condition by setting
\[
	\mathbf{u}(-2,y) = \left(\sin\left(\pi \frac{2+y}{4}\right),0\right)^\top \quad \text{for}\quad y\in(-2,2).
\]
On $\{30\}\times(-2,2)$, we choose an outlet boundary condition, that is, we use the homogeneous Neumann condition
\[
\nabla \mathbf{u}\cdot \mathbf{n}+p\mathbf{n} =0.
\]
This problem does not coincide with our model problem~\eqref{eq:stokes} as we have a Neumann boundary condition. Using the Neumann condition, the average of the pressure is uniquely solvable in $L^2(\Omega)$. So, we omit the condition on the average pressure.
We solve the resulting system as proposed in Section~\ref{sec:4}. The only difference is that we now do not average the overall pressure, that is, we no longer enforce $C_\Pi\, \underline{\mathbf x}_\Pi=0$.

\begin{table}[th]
\scriptsize
	\newcolumntype{L}[1]{>{\raggedleft\arraybackslash\hspace{-1em}}m{#1}}
	\centering
	\renewcommand{\arraystretch}{1.25}
	\begin{tabular}{l|L{1em}L{1.8em}|L{1em}L{1.8em}|L{1em}L{1.8em}|L{1em}L{1.8em}|L{1em}L{1.8em}}
		\toprule
		\multicolumn{1}{l}{$\ell\;\;\diagdown\;\;\degree$\hspace{-1.8em}\;}
		& \multicolumn{2}{c|}{2}
		& \multicolumn{2}{c|}{3}
		& \multicolumn{2}{c|}{4}
		& \multicolumn{2}{c|}{5}
		& \multicolumn{2}{c}{6} \\
		& it & $\kappa$
		& it & $\kappa$
		& it & $\kappa$
		& it & $\kappa$
		& it & $\kappa$ \\
		\midrule
		$2$  & $11$ & $ 4.4$ & $11$ & $ 5.3$ & $11$ & $ 6.0$ & $12$ & $ 6.6$ & $12$ & $7.1$ \\
		$3$  & $12$ & $ 6.0$ & $12$ & $ 6.9$ & $13$ & $ 7.7$ & $13$ & $ 8.4$ & $13$ & $9.0$ \\
		$4$  & $13$ & $ 7.7$ & $13$ & $ 8.8$ & $13$ & $ 9.6$ & $13$ & $10.4$ & $14$ & $11.1$ \\
		$5$  & $14$ & $ 9.6$ & $14$ & $10.8$ & $14$ & $11.9$ & $14$ & $12.7$ & $14$ & $13.5$ \\
		\bottomrule
	\end{tabular}
	\captionof{table}{Iteration counts (it) and condition numbers $\kappa$, rectangle domain
	  \label{tab:Rectangle}}
\end{table}
    In Table~\ref{tab:Rectangle}, we present iteration counts and estimated condition numbers for various refinement levels $\ell$ and spline degrees $\degree$. These numbers behave similar to those presented in Tables~\ref{tab:QuarterAnnulus} and~\ref{tab:YetiFoot}. In Figure~\ref{fig:flow}, we present a reconstruction of the solution for $\ell=4$ and $\degree = 2$. Note that we changed the sign of the pressure such that it is positive.
\begin{figure}[htb] 
	\centering
	\includegraphics[height=3.2cm]{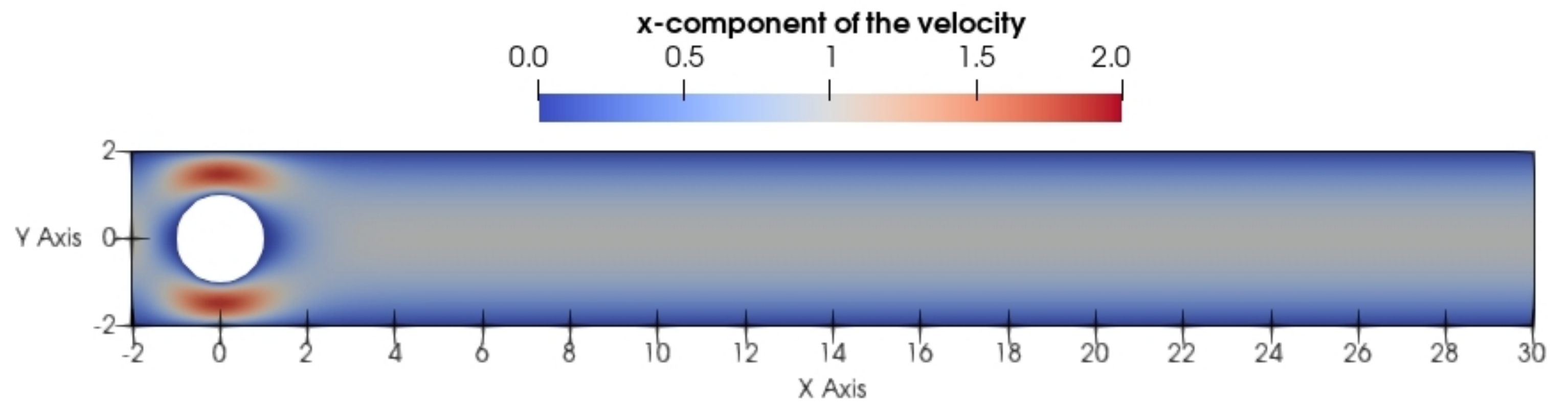}
    \qquad
	\includegraphics[height=3.2cm]{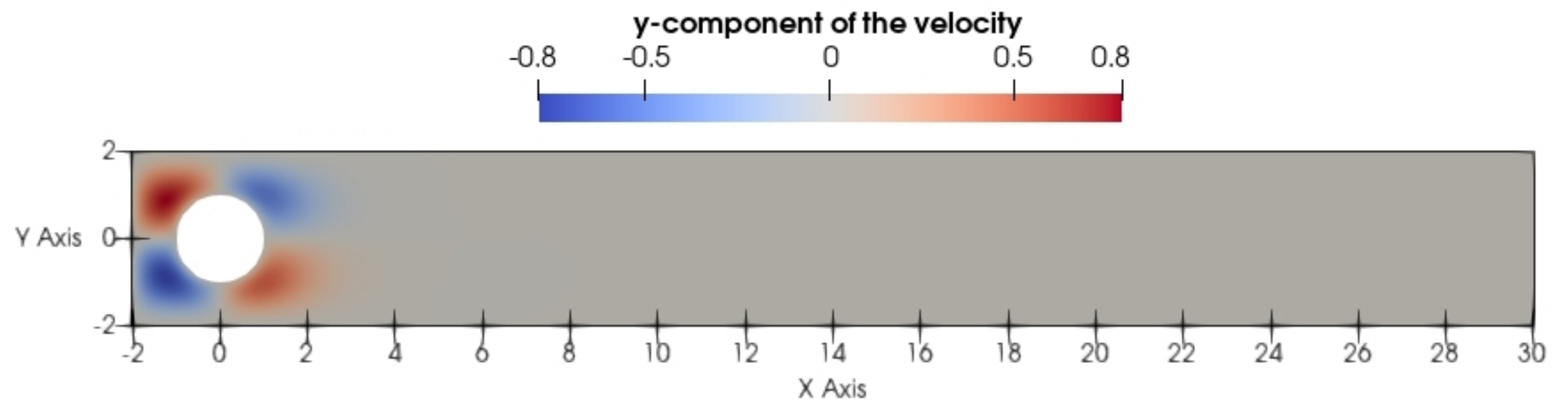}
    \qquad
	\includegraphics[height=3.2cm]{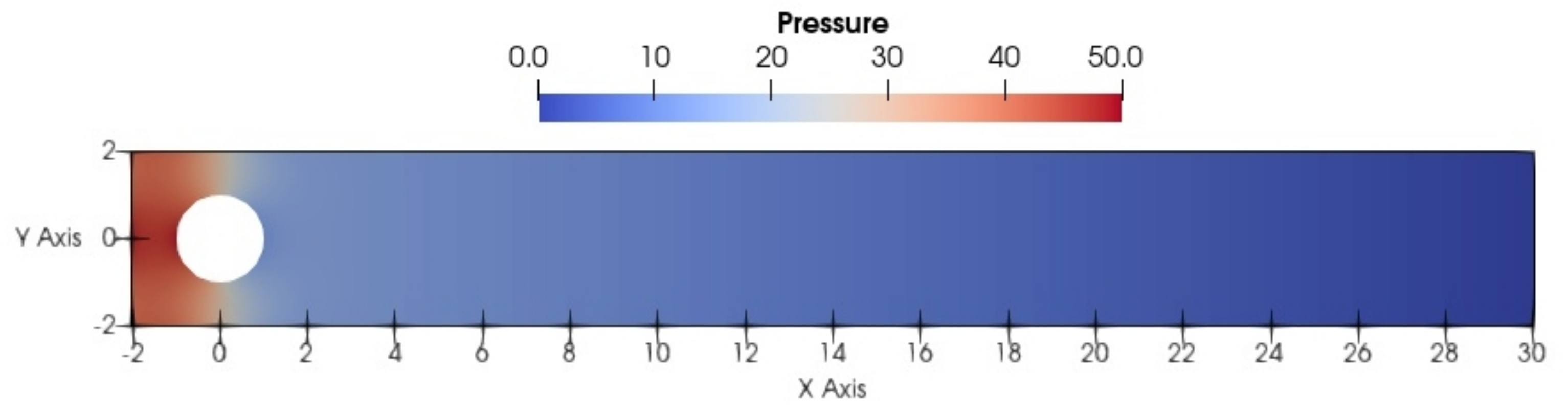}
	\caption{Reconstructed solution for $\ell=4$ and $\degree = 2$}
	\label{fig:flow}
\end{figure}

\setcounter{section}{0}
\renewcommand{\thesection}{\Alph{section}} 

\section{Appendix}
\label{sec:localCorrectionOperator}

Before we give a proof of Lemma~\ref{lamma:FortinOp}, we construct an
operator~$\mathbf{\Pi}_{\widetilde F}$ that
is similar to the Fortin operator. This construction only requires a few basis functions
per patch and is best understood as a variation of well-known techniques in the area
of finite elements, where the patches $\Omega^{(k)}$ play the role of elements.

\begin{lemma}\label{lem:localCorrectionOperator}
	There exists an operator $\mathbf{\Pi}_{\widetilde F}: [H^1_0(\Omega)]^2\rightarrow \mathbf{V}$ such that
	\begin{equation}\label{eq:pi2:stab}
				|\mathbf{\Pi}_{\widetilde F} \mathbf{u}|_{H^1(\Omega)}^2
				\le \widetilde c_F
				\sum_{k=1}^K \big( H_k^{-2} \|\mathbf{u}\|_{L^2(\Omega^{(k)})}^2 + |\mathbf{u}|_{H^1(\Omega^{(k)})}^2 \big)
				\quad\forall\,\mathbf{u}\in [H^1_0(\Omega)]^2
	\end{equation}
	and
	\begin{equation}\label{eq:pi2:divpres0}
				(\nabla \cdot (I-\mathbf{\Pi}_{\widetilde F})\mathbf{u}, p_1)_{L^2(\Omega)}=0
				\quad\forall\,	\mathbf{u}\in [H^1_0(\Omega)]^2
				\quad \mbox{and}\quad \forall\,p_1\in Q_1,
	\end{equation}
	hold, where $\widetilde c_F>0$ only depends on the constants
	from Assumptions~\ref{ass:nabla} and~\ref{ass:normals}.
\end{lemma}
\begin{proof}
	For each interface $\Gamma^{(k,\ell)}$ with pre-image $\widehat{\Gamma}^{(k,\ell)}:=
	\mathbf G_k^{-1}(\Gamma^{(k,\ell)})$,
	we define a function $\psi^{(k,\ell)}$ as follows. 
	If $\widehat{\Gamma}^{(k,\ell)}=\{1\}\times[0,1]$, we
	define $\psi^{(k,\ell)}$ on patch $\Omega^{(k)}$ by
	\begin{equation}\label{eq:corrfunc1}
			\psi^{(k,\ell)}( x ) = 
				\widehat\varphi(\mathbf G_k^{-1}(x))\,
					\mathbf n^{(k)}(\overline{x}^{(k,\ell)})
			\quad\mbox{with}\quad
			\widehat\varphi(\xi_1,\xi_2):=			
			\xi_1 \xi_2 (1-\xi_2),
	\end{equation}
	where $\mathbf n^{(k)}$ and $\overline{x}^{(k,\ell)}$ are as in Assumption~\ref{ass:normals}.
	If $\widehat{\Gamma}^{(k,\ell)}$ is one of the other sides, the
	function $\widehat\varphi$ is rotated around the center of the unit square
	such that it is non-zero on $\widehat{\Gamma}^{(k,\ell)}$ and zero
	on the other sides. On the patch $\Omega^{(\ell)}$, we define
	\begin{equation}\label{eq:corrfunc2}
			\psi^{(k,\ell)}:=-\psi^{(\ell,k)}.
	\end{equation}
	On all other patches, we set $\psi^{(k,\ell)}:=0$.
	Note that this construction guarantees that $\psi^{(k,\ell)}$ is continuous. (The negative sign in~\eqref{eq:corrfunc2} is due to
	$\mathbf n^{(\ell)}=-\mathbf n^{(k)}$.)
	Since the spaces $\widehat{\mathbf V}^{(k)}$ and $\widehat{\mathbf V}^{(\ell)}$
	contain quadratic functions, they contain $\widehat\varphi$. Thus,
	we have $\psi^{(k,\ell)}\in \mathbf V$.
	
	By this construction, we obtain
	\begin{equation}\label{eq:40a}
			(\psi^{(k,\ell)} \cdot \mathbf n^{(k)},1)_{L^2(\Gamma^{(k,\ell)})}\not=0
	\end{equation}
	and
	\begin{equation}\label{eq:local}
	\psi^{(k,\ell)}|_{\Gamma^{(r,s)}}=0
	\quad\mbox{for}\quad
	\{k,\ell\}\not=\{r,s\}.
	\end{equation}
	Next, we define the projector $\mathbf{\Pi}_{\widetilde F}$.
	Let $\mathbf u$ be arbitrary but fixed. We define
	\[
			\mathbf{\Pi}_{\widetilde F} \mathbf u :=
			\sum_{k=1}^K
			\sum_{\ell\in \mathcal N_\Gamma(k), \ell>k}
			\frac{(\mathbf u \cdot \mathbf n^{(k)},1)_{L^2(\Gamma^{(k,\ell)})}}
			{(\psi^{(k,\ell)} \cdot \mathbf n^{(k)},1)_{L^2(\Gamma^{(k,\ell)})}}
			\psi^{(k,\ell)}.
	\]	
	Using this definition,~\eqref{eq:40a} and~\eqref{eq:local}, we immediately
	obtain
	\[
		((\mathbf u-\mathbf{\Pi}_{\widetilde F}\mathbf u) \cdot \mathbf n^{(k)},1)_{L^2(\Gamma^{(k,\ell)})}
			 = 0
	\]
	for all $k$ and all $\ell\in\mathcal N_\Gamma(k)$. This
	immediately yields
	\[
				((\mathbf u-\mathbf{\Pi}_{\widetilde F}\mathbf u) \cdot \mathbf n^{(k)},1)_{L^2(\partial\Omega^{(k)})}
				=0,
	\]
	and by integration by parts further~\eqref{eq:pi2:divpres0}, which
	finishes the first part of the proof.
	Next, we estimate the $H^1$-seminorm of $\mathbf{\Pi}_{\widetilde F}u$.
	Using the triangle inequality, $|\mathcal N_\Gamma(k)|\le 4$
	and the bounded support of $\psi^{(k,\ell)}=-\psi^{(\ell,k)}$, we obtain
	\begin{equation}\label{eq:append1}
	\begin{aligned}
		|\mathbf{\Pi}_{\widetilde F}\mathbf u|_{H^1(\Omega)}^2
		&=
		\left|
			\sum_{k=1}^K
			\sum_{\ell\in \mathcal N_\Gamma(k), \ell>k}
			\frac{(\mathbf u \cdot \mathbf n^{(k)},1)_{L^2(\Gamma^{(k,\ell)})}}
			{(\psi^{(k,\ell)} \cdot \mathbf n^{(k)},1)_{L^2(\Gamma^{(k,\ell)})}}
			\psi^{(k,\ell)}
		\right|_{H^1(\Omega)}^2\\
		&\le 4
		\sum_{k=1}^K
			\sum_{\ell\in \mathcal N_\Gamma(k)}
			\frac{(\mathbf u \cdot \mathbf n^{(k)},1)_{L^2(\Gamma^{(k,\ell)})}^2}
			{(\psi^{(k,\ell)} \cdot \mathbf n^{(k)},1)_{L^2(\Gamma^{(k,\ell)})}^2}
    \left|
			\psi^{(k,\ell)}
		\right|_{H^1(\Omega^{(k)})}^2.
	\end{aligned}
	\end{equation}
	In the remainder of this proof, we write $a\lesssim b$ (or $b\gtrsim a$) if
	there is a constant $c>0$ that only depends on the constants from
	Assumptions~\ref{ass:nabla} and \ref{ass:normals} such that $a\le c\,b$. 
	Using Assumptions~\ref{ass:nabla} and \ref{ass:normals} and~\eqref{eq:corrfunc1}, we obtain
	\begin{equation}\label{eq:append2}
	\begin{aligned}
			&|(\psi^{(k,\ell)} \cdot \mathbf n^{(k)},1)_{L^2(\Gamma^{(k,\ell)})}|
			 \gtrsim
			H_k
			|(\widehat \varphi,1)_{L^2(\widehat\Gamma^{(k,\ell)})}|
			=H_k\sqrt{1/30}.
	\end{aligned}
	\end{equation}
	Using~\eqref{eq:geoequiv} and~\eqref{eq:corrfunc1}, we also obtain
	\begin{equation}\label{eq:append3}
			|\psi^{(k,\ell)} |_{H^1(\Omega^{(k)})}
			 \lesssim
			| \widehat \varphi |_{H^1(\widehat\Omega)}
			=\sqrt{13/90}.
	\end{equation}
		Using a combination of~\eqref{eq:append1}, \eqref{eq:append2}
		and~\eqref{eq:append3}, the Cauchy-Schwarz inequality,
		Assumption~\ref{ass:nabla}, $\|1\|_{L^2(\Gamma^{(k,\ell)})}^2 \lesssim H_k \|1\|_{L^2(\widehat\Gamma^{(k,\ell)})}^2$, and a standard estimate for the trace yield 
	\begin{align*}
		&|\mathbf{\Pi}_{\widetilde F}\mathbf u|_{H^1(\Omega)}^2
		\lesssim
		\sum_{k=1}^K
			\sum_{\ell\in \mathcal N_\Gamma(k)}
			H_k^{-2}\,
			(\mathbf u \cdot \mathbf n^{(k)},1)_{L^2(\Gamma^{(k,\ell)})}^2
		\\&\quad\lesssim
		\sum_{k=1}^K
			\sum_{\ell\in \mathcal N_\Gamma(k)}
			H_k^{-1}\,
			\| \mathbf u \|_{L^2(\Gamma^{(k,\ell)})}^2
			\lesssim
		\sum_{k=1}^K
			\sum_{\ell\in \mathcal N_\Gamma(k)}
			\| \mathbf u\circ\mathbf G_k \|_{L^2(\widehat\Gamma^{(k,\ell)})}^2
		\\& \quad \lesssim
			\sum_{k=1}^K
			(
			| \mathbf u\circ\mathbf G_k |_{H^1(\widehat\Omega^{(k)})}^2+
			\| \mathbf u\circ\mathbf G_k \|_{L^2(\widehat\Omega^{(k)})}^2)
	\end{align*}
	The estimate \eqref{eq:geoequiv} finishes the proof.
\end{proof}

Note that the operator from Lemma~\ref{lem:localCorrectionOperator} fails to meet the
conditions for a Fortin operator due to the additional term of the form
$H_k^{-2}\|\mathbf u\|_{L^2(\Omega^{(k)})}^2$. For the proof of the desired error
bound, we need an approximation error estimate that only needs to decrease with
the patch size $H_k$, not with the grid size $h_k$. For the construction of such
an approximation error estimate, we use a Scott-Zhang operator, which relies on
Poincaré estimates. In the following, we verify that the Poincaré constant for
the subdomains $\mathcal S^{(j)}$ from~\eqref{eq:Sjdef} only depends on the
constants from Assumptions~\ref{ass:neighbors} and~\ref{ass:nabla}.
\begin{lemma}\label{lem:poincare}
		For all $j=1,\ldots,J$, we have
		\[
		  \inf_{c\in\mathbb R} \|u-c\|_{L^2(\mathcal S^{(j)})}
		    \le \widetilde c_p\;
		    \mathrm{diam}\, \mathcal S^{(j)}\; |u|_{H^1(\mathcal S^{(j)})}
		    \quad\foralls u\in H^1_0(\Omega),
		\]
		where
		\begin{equation}\label{eq:Sjdef}
		\mathcal S^{(j)}:=
			\bigcup_{k\in\mathcal N_x(j)} \overline{ \Omega^{(k)} },
			\qquad j=1,\ldots,J.
		\end{equation}
		and $\widetilde c_p$ only depends on the constants from
		the Assumptions~\ref{ass:neighbors} and~\ref{ass:nabla}.
\end{lemma}
\begin{proof}
		Let $j$ be arbitrary but fixed. Let $N:=\mathcal N_x(j)$ be the
		number of patches adjacent to the vertex $x_j$.
		From Assumptions~\ref{ass:neighbors} and~\ref{ass:nabla}, we know $3\le N\le C_2$.
		Let $\Omega^{(k_1)},\ldots,\Omega^{(k_N)}$ be the patches adjacent
		to $x_j$, enumerated in counter-clockwise ordering. $\mathcal S^{(j)}$
		is the union of these patches. Let $\widetilde{\mathcal S}$ be a pre-image
		of $\mathcal S^{(j)}$ consisting of rhombi $\widetilde \Omega_n$, $n=1,\ldots,N$,
		of size $1$, arranged
		as depicted in Figure~\ref{fig:poincare}. Let 
		$\mathbf T_n:\widehat\Omega \rightarrow \widetilde \Omega_n$
		be the canonical linear maps with positive Jacobi-determinant 
		and such that $\mathbf G_{k_n}(\mathbf T_n^{-1}(\widetilde x_j))
		=x_j$, where $\widetilde{x}_j$ is the common vertex
		of the rhombi, see Figure~\ref{fig:poincare}.
		\begin{figure}[h]
	  \begin{center}
		\includegraphics{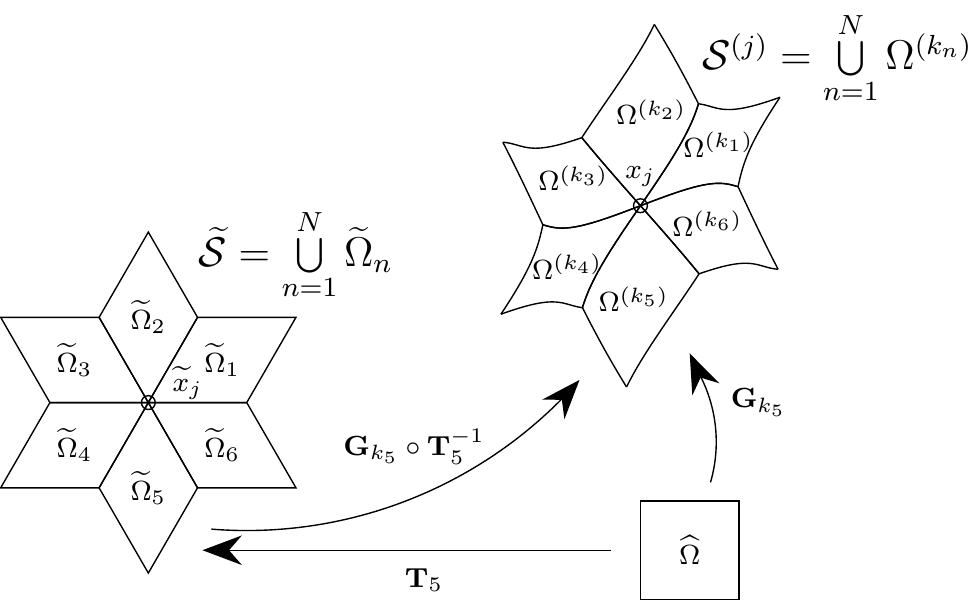}
	  \end{center}
	  \caption{\label{fig:poincare} $\widetilde{\mathcal S}$, the pre-image of $\mathcal S^{(j)}$
	  and corresponding mappings}
	  \end{figure}
		On $\widetilde{\mathcal S}$,  a Poincaré inequality holds
		\begin{equation}\label{eq:lem:poincare}
				\inf_{c\in\mathbb R} \|u-c\|_{L^2(\widetilde{\mathcal S})}
				\le
				\widetilde c_P
				|u|_{H^1(\widetilde{\mathcal S})}
				\quad\foralls u\in H^1(\widetilde{\mathcal S}),
		\end{equation}
		where the constant $\widetilde c_P$ only depends on $N$,
		which is well-bounded due to Assumption~\ref{ass:neighbors}.
		The statement~\eqref{eq:lem:poincare} can be transferred to
		$\mathcal S^{(j)}$ by applying the same arguments as
		for~\eqref{eq:poincare}. The
		constants only depend on the Jacobians of $\mathbf G_{k_n}$ and
		$\mathbf T_n$, which are bounded due to Assumptions~\ref{ass:neighbors}
		and~\ref{ass:nabla}. This finishes the proof.
\end{proof}

Let
\[
\mathbf{V}_1:=\{ \mathbf u \in [H^1_0(\Omega)]^2 : \mathbf u\circ \mathbf G_k \mbox{ bilinear for }k=1,\ldots,K \} \subset \mathbf{V}
.
\]
We choose the function values at the vertices $x_1,\ldots,x_J$, which are the corners of the patches not located on the (Dirichlet) boundary, as the degrees of freedom. Based on this choice of degrees of freedom and the corresponding nodal basis, we define a Scott-Zhang projector $\mathbf \Pi_{SZ} : [H^1_0(\Omega)]^2 \rightarrow \mathbf V_1$  (cf. Ref.~\refcite{ScottZhang:1990}).
Using the Poincaré inequalities~\eqref{eq:poincare} and Lemma~\ref{lem:poincare} and the Friedrich's inequality~\eqref{eq:friedrichs}, using the same arguments as in Ref.~\refcite{ScottZhang:1990} we obtain 
\begin{equation}\label{eq:pi1:stab}
		|\mathbf{\Pi}_{SZ} \mathbf{u}|_{H^1(\Omega)} \le  c_S\, |\mathbf{u}|_{H^1(\Omega)} \quad \forall\, \mathbf{u}\in [H^1_0(\Omega)]^2
\end{equation}
and
\begin{equation}\label{eq:pi1:approx3}
		\sum_{k=1}^K H_k^{-2} \|(I-\mathbf{\Pi}_{SZ}) \mathbf{u}\|_{L^2(\Omega^{(k)})}^2
		\leq c_S\,|\mathbf{u}|_{H^1(\Omega)}^2 \quad \forall\, \mathbf{u}\in [H^1_0(\Omega)]^2,
\end{equation}
where the constant $c_S>0$ only depends on the constants from the Assumptions~\ref{ass:neighbors} and~\ref{ass:nabla}.

We continue by giving a constructive proof for the existence of a Fortin operator.
\begin{proofof}{of Lemma \ref{lamma:FortinOp}}
We define $\mathbf{\Pi}:[H^1_0(\Omega)]^2\rightarrow \mathbf{V}$ as
\[
\mathbf{\Pi} := \mathbf{\Pi}_{SZ} + \mathbf{\Pi}_{\widetilde{F}} (I- \mathbf{\Pi}_{SZ}),
\]
where $\mathbf{\Pi}_{SZ}$ is the Scott-Zhang projector and $\mathbf{\Pi}_{\widetilde{F}}$ is the operator
from Lemma~\ref{lem:localCorrectionOperator}.
Using the triangle inequality and \eqref{eq:pi2:stab}, we obtain
	\begin{align*}
		&|\mathbf{\Pi} \mathbf{u}|_{H^1(\Omega)}^2 = |\mathbf{\Pi}_{SZ} \mathbf{u} + \mathbf{\Pi}_{\widetilde{F}}(I-\mathbf{\Pi}_{SZ})\mathbf{u}|_{H^1(\Omega)}^2 \\
		&  \le 2 |\mathbf{\Pi}_{SZ} \mathbf{u}|_{H^1(\Omega)}^2 + 2 |\mathbf{\Pi}_{\widetilde{F}}(I-\mathbf{\Pi}_{SZ})\mathbf{u}|_{H^1(\Omega)}^2 \\
		&  \le 2 |\mathbf{\Pi}_{SZ}\mathbf{u}|_{H^1(\Omega)}^2 + 2 c_{\widetilde{F}} \sum_{k=1}^K\left( H_k^{-2} \|(I-\mathbf{\Pi}_{SZ})\mathbf{u}\|_{L^2(\Omega^{(k)})}^2 + |(I-\mathbf{\Pi}_{SZ})\mathbf{u}|_{H^1(\Omega^{(k)})}^2 \right).
    \end{align*}
    We now use~\eqref{eq:pi1:stab} and~\eqref{eq:pi1:approx3} to show~\eqref{eq:pi:stabtmp}.
	It remains to show \eqref{eq:pi:divpres0}. We have
	\[
			(\nabla \cdot (I-\mathbf{\Pi}) \mathbf{u},p_1)_{L^2(\Omega)} = (\nabla \cdot (I-\mathbf{\Pi}_{\widetilde{F}})\underbrace{(I-\mathbf{\Pi}_{SZ}) \mathbf{u}}_{\displaystyle \mathbf{w}:=}, p_1)_{L^2(\Omega)} = 0
	\]
	for all $p_1\in Q_1$
	by applying \eqref{eq:pi2:divpres0} to $\mathbf{w}$, which concludes the proof.
\end{proofof}

\begin{proofof}{of Lemma~\ref{lem:supmat}}
	Since the statements~\eqref{eq:supmat0} and \eqref{eq:supmat1} can be found in the literature, cf. Ref.~\refcite{fortin1991mixed}, Chapter~II, §~1.1, we only show~\eqref{eq:supmat2}.
 	Let $\underline w \in W_2$ be such that it maximizes
 	$\sup_{\underline w\in W_2} \frac{(B\underline w,\underline\lambda)_{\ell^2}}{\|\underline w\|_A}$.
 	Observe that $\underline w$ is only defined up to scaling. So, we introduce
 	the constraint $\underline w^\top A \underline w=1$. The first order optimality
 	system for the minimizer then reads as follows:
 	\begin{equation}\label{eq:supmat:2}
 	\begin{aligned}
 		 B^\top \underline \lambda + \xi A \underline w + C^\top \underline \mu & = 0 \\
 		 \underline w^\top A \underline w &=1 \\
 		 D \underline \nu = 0 \quad\Rightarrow\quad \underline \nu^\top C \underline w &=0\\
 		 D \underline \mu &= 0,
 	\end{aligned}
 	\end{equation}
 	where $\xi \in \mathbb R$ and $\underline \mu\in \mathbb R^{m_2}$ are the Lagrange multipliers.
 	The specific form in the third line in~\eqref{eq:supmat:2} is obtained by the fact that
 	we may only consider derivatives in the feasible directions.
 	By multiplying the first line in~\eqref{eq:supmat:2} from left with $\underline w^\top$, we
 	obtain using the second line in~\eqref{eq:supmat:2}
 	\[
 			\underline w^\top B^\top \underline \lambda + \xi  + \underline w^\top C^\top \underline \mu  = 0.
 	\]
 	Since $D\underline \mu=0$, we know from the third line in~\eqref{eq:supmat:2}
 	that $ \underline w^\top C^\top \underline \mu=0$. This shows
 	$\xi = - \underline w^\top B^\top \underline \lambda= -\underline\lambda^\top B\underline w$.
 	The third line in~\eqref{eq:supmat:2} is satisfied if and only
 	if $C \underline w + D^\top \underline \rho = 0$ for some $\underline \rho$.
 	From this line, the first line in~\eqref{eq:supmat:2} and the fourth line in~\eqref{eq:supmat:2},
 	we obtain
 	\[
 		-
 		\begin{pmatrix}
 			\xi \underline w \\ \underline \mu\\ \xi \underline \rho
 		\end{pmatrix}
 		=
 		\begin{pmatrix}
 			 A & C^\top & 0\\
 			 C & 0 & D^\top \\
 			 0 & D & 0
 		\end{pmatrix}^{-1} 		
 		\begin{pmatrix}
 			B^\top\underline \lambda \\ 0 \\0
 		\end{pmatrix}.
 	\]
 	By multiplying this from left with $\begin{pmatrix} \underline\lambda^\top B & 0 & 0\end{pmatrix}$, we obtain
 	\[
 	-\xi \underline\lambda^\top B \underline w
 	=
 	\begin{pmatrix} \underline\lambda^\top B & 0& 0\end{pmatrix}
 	\begin{pmatrix}
 			 A & C^\top & 0\\
 			 C & 0 & D^\top \\
 			 0 & D & 0
 		\end{pmatrix}^{-1} 		
 		\begin{pmatrix}
 			B^\top\underline \lambda \\ 0\\ 0
 		\end{pmatrix}=
 		\|\underline\lambda\|_{M_2}^2.
 	\]
 	Using $\|\underline w\|_A^2=1$ and $\xi = -\underline\lambda^\top B \underline w$,
 	we obtain
 	\[
 		\frac{(\underline\lambda^\top B \underline w)^2}{\|\underline w\|_A^2}
 		= \|\underline\lambda\|_{M_1}^2,
 	\]
 	which finishes the proof. 
\end{proofof}

\bibliographystyle{ws-m3as}
\bibliography{bibliography}

\end{document}